\documentclass{amsart}
\usepackage{graphicx}
\usepackage{amssymb}
\usepackage{amsmath}
\usepackage{amsthm}
\usepackage{subfigure}
\usepackage{enumerate}
\usepackage{bbm}
\usepackage{xcolor}

\newtheorem{thm}{Theorem}[section]
\newtheorem{cor}[thm]{Corollary}

\newtheorem{lem}[thm]{Lemma}
\newtheorem{prop}[thm]{Proposition}
\theoremstyle{definition}
\newtheorem{defn}[thm]{Definition}
\newtheorem{rem}[thm]{Remark}

\newtheorem*{defn*}{Definition}
\newtheorem*{rems*}{Remarks}
\newtheorem*{rem*}{Remark}

\newtheorem{alg}[thm]{Algorithm}

\numberwithin{equation}{section}

\newcommand{\Eq}{{\text{E}}}

\newcommand{\M}{{M}}

\newcommand{\Css}{{\text{CSS}}}

\newcommand{\C}{{C}}

\renewcommand{\d}{\mathrm{d}}

\newcommand{\minm}{\mathbbm{m}}
\newcommand{\maxm}{\mathbb{M}}

\newcommand{\overarc}[2]{\underline{#1\frown #2}}
\newcommand{\p}{p}

\begin{document}

\title[The geometry of the Wigner caustic] {The Geometry of the Wigner Caustic and a~Decomposition of a Curve into Parallel Arcs}
\author{Wojciech Domitrz, Micha\l{} Zwierzy\'nski}
\address{Faculty of Mathematics and Information Science\\
Warsaw University of Technology\\
ul. Koszykowa 75, 00-662 Warszawa\\
Poland
\\}

\email{domitrz@mini.pw.edu.pl, zwierzynskim@mini.pw.edu.pl}
\thanks{The work of W. Domitrz and M. Zwierzy\'nski was partially supported by NCN grant no. DEC-2013/11/B/ST1/03080. }



\begin{abstract}
In this paper we study global properties of the Wigner caustic  of parameterized closed planar curves. We find new results on its geometry and singular points. In particular, we consider the Wigner caustic of rosettes, i.e. regular closed parameterized curves with non-vanishing curvature.

We present a decomposition of a curve into parallel arcs to describe smooth branches of the Wigner caustic. By this construction we can find the number of smooth branches, the rotation number, the number of inflexion points and the parity of the number of cusp singularities of each branch. We also study the global properties of the Wigner caustic on shell (the branch of the Wigner caustic connecting two inflexion points of a curve). We apply our results to whorls -- the important object to study the dynamics of a quantum particle in the optical lattice potential.
\end{abstract}

\maketitle

\section{Introduction}

In 1932  Eugene Wigner introduced the celebrated Wigner function   to study quantum corrections to classical statistical mechanics (\cite{Wigner}).This function relates the wavefunction that appears in Schrödinger's equation to a probability distribution in phase space. The Wigner function of a pure state  is defined in the following way
$$
{\mathcal W}_{\psi}(p,q)=\frac{1}{\pi\hbar}\int^{\infty}_{-\infty} \psi^*(q-\zeta)\psi(q+\zeta)\exp{(2ip\zeta/\hbar)} \, \d\zeta,
$$
where $(p,q)\in \mathbb R^2$ are momentum and position, and $\psi\in L^2_{\mathbb C}(\mathbb R)$ is the wavefunction.
In  \cite{B1} M. Berry studied  the semiclassical limit of Wigner's phase-space representation of quantum states. He proved that for $1$-dimensional systems, that correspond to  smooth (Lagrangian) curves $M$ in the phase space $(\mathbb R^2, \omega=\d p \wedge \d q)$,  the semiclassical limit of the Wigner function of the classical correspondence $\M$ of a pure quantum state takes on high values at points in a neighborhood of $\M$ and also in a neighborhood of a singular closed curve, which is called the Wigner caustic of $\M$  or the Wigner catastrophe (see \cite{B1, OH, DMR1, BW} for details). Geometrically the Wigner caustic of a planar curve $M$ is the locus of midpoints of chords connecting points on $\M$ with parallel tangent lines (\cite{B1, DMR1, DR1, OH}). It is caustic of a
Lagrangian map defined in the following way (see \cite{OH, DR1,DMR1} for details).

For the  the canonical symplectic form $\displaystyle\omega=\mathrm{d}p\wedge \mathrm{d}q$ on $\mathbb{R}^2$ the map $\flat: T\mathbb{R}^{2}\ni v\mapsto \omega(v, \cdot)\in T^\ast\mathbb{R}^{2}$ is an isomorphism between the bundles $T\mathbb{R}^{2}$ and $T^\ast\mathbb{R}^{2}$.  Then $\displaystyle\dot{\omega}=\flat^\ast\mathrm{d}\alpha=\mathrm{d}\dot{p}\wedge \mathrm{d}q+\mathrm{d}p\wedge\mathrm{d}\dot{q}$ is a symplectic form on $T\mathbb{R}^{2}$, where  $\alpha$ be the canonical Liouville $1$-form on $T^\ast\mathbb{R}^2$
 The linear diffeomorphism $\Phi_{\frac{1}{2}}:\mathbb{R}^{2}\times\mathbb{R}^{2}\to T\mathbb{R}^2=\mathbb{R}^{2}\times\mathbb{R}^{2},$
\begin{align*}
\Phi_{\frac{1}{2}}(p^+, q^+, p^-, q^-)=(p,q,\dot{p},\dot{q})=\frac{1}{2}\left(p^++p^-, q^++q^-,p^+-p^-, q^+-q^-\right)
\end{align*}
pullbacks  the symplectic form $\dot{\omega}$ on $T\mathbb{R}^{2}$ the canonical symplectic form $\frac{1}{2}(\pi_+^\ast\omega-\pi_-^\ast\omega)$ on the product $\mathbb{R}^{2}\times\mathbb{R}^{2}$, where $\pi_+, \pi_-:\mathbb{R}^{2}\times\mathbb{R}^{2}\to\mathbb{R}^{2}$ are the projections on the first and on the second component, respectively. If $M$ is a smooth regular planar curve then $M$ is an immersed Lagrangian submanifold of $(R^2, \omega)$. Hence $\Phi_{\frac{1}{2}}(M\times M)$ is an immersed  Lagrangian submanifold of $(T\mathbb{R}^{2}, \dot{\omega})$. Let $\pi_1, \pi_2:T\mathbb{R}^2=\mathbb{R}^{2}\times\mathbb{R}^{2}\to\mathbb{R}^{2}$ be the projections on the first and on the second component, respectively. Then $\pi_1$ and $\pi_2$ define Lagrangian fibre bundles with the symplectic structure $\dot{\omega}$. Then the caustic of the Lagrangian map (the set of its critical values) $\pi_1\circ\Phi_{\frac{1}{2}}\big|_{M\times M}$ is the Wigner caustic \cite{OH, DR1, DMR1, CDR1, CDR2}. On the other hand the Lagrangian map  $\pi_2\circ\Phi_{\frac{1}{2}}\big|_{M\times M}$ is the secant map of $M$ \cite{DRZ1}.  If $M$ is (locally) described as 
$$M=\left\{(p,q)\in \mathbb R^2\,|\, p=\frac {\d S}{\d q}(q)\right\}$$ 
then the generating family of the Lagrangian submanifold $\Phi_{\frac{1}{2}}(M\times M)$ has the following form
$$
F(p,q,\beta)=\frac{1}{2}\left(S(q+\beta)-S(q-\beta)\right)-p\beta.
$$
The front of the Legendrian submanifold of the contact manifold $(T{\mathbb R}^2\times \mathbb R, \d z+\flat^{\ast}\alpha)$
generated by $F$ is a singular 2-dimensional improper affine sphere, where $z$ is a coordinate on $\mathbb R$. The caustic of this front is composed of the curve $M$ and its Wigner caustic. Hence the geometry of the Wigner caustics provides information on singularities of improper affine spheres. In Fig. \ref{FigIntroIASProjection} we present a non-convex planar curve with its Wigner caustic and in Fig. \ref{FigIntroIASBeans} we show the improper affine sphere generated by $M$ in the construction described above (see \cite{CDR1, CDR2} for details).  

\begin{figure}[h]
\centering
\includegraphics[scale=0.30]{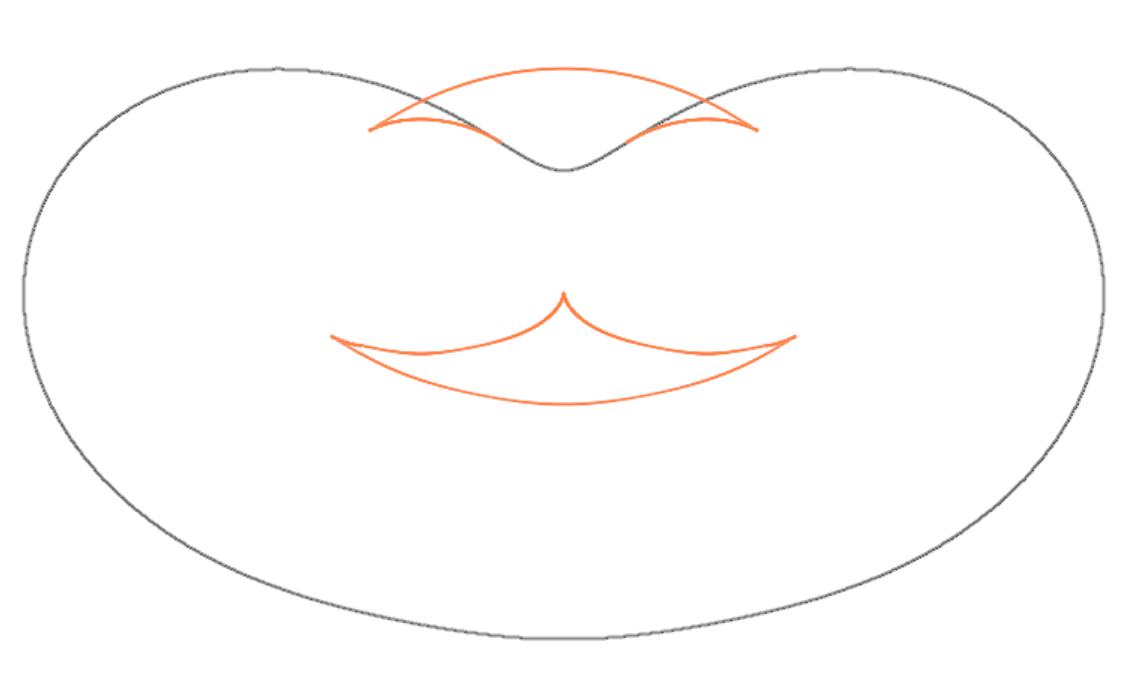}
\caption{A non-convex curve together with its Wigner caustic}
\label{FigIntroIASProjection}
\end{figure}

\begin{figure}[h]
\centering
\includegraphics[scale=0.15]{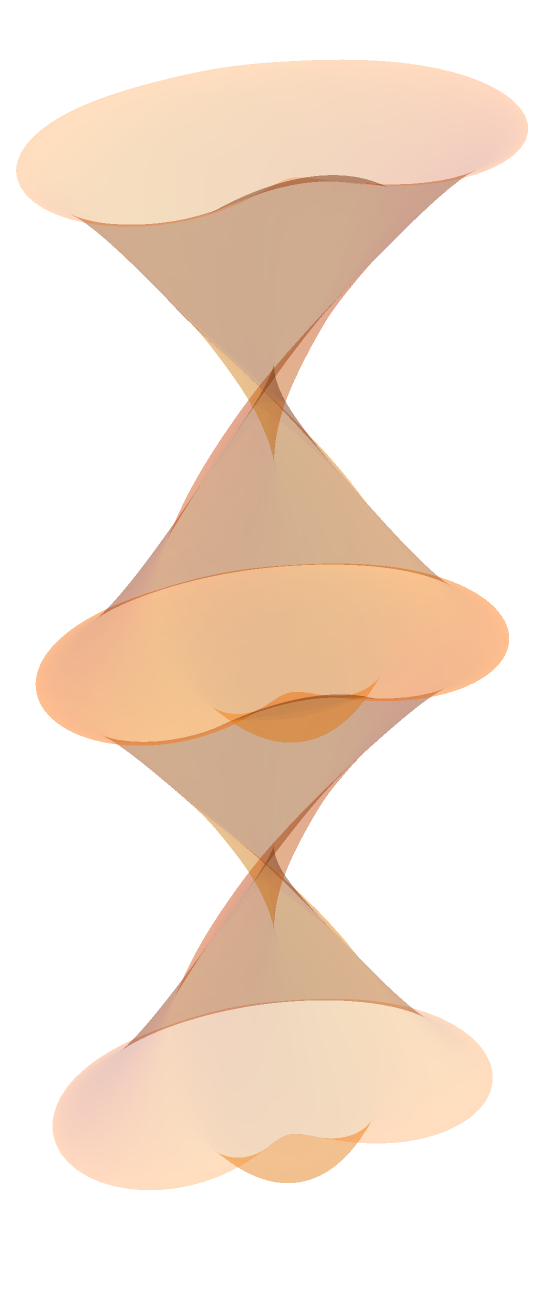}
\includegraphics[scale=0.15]{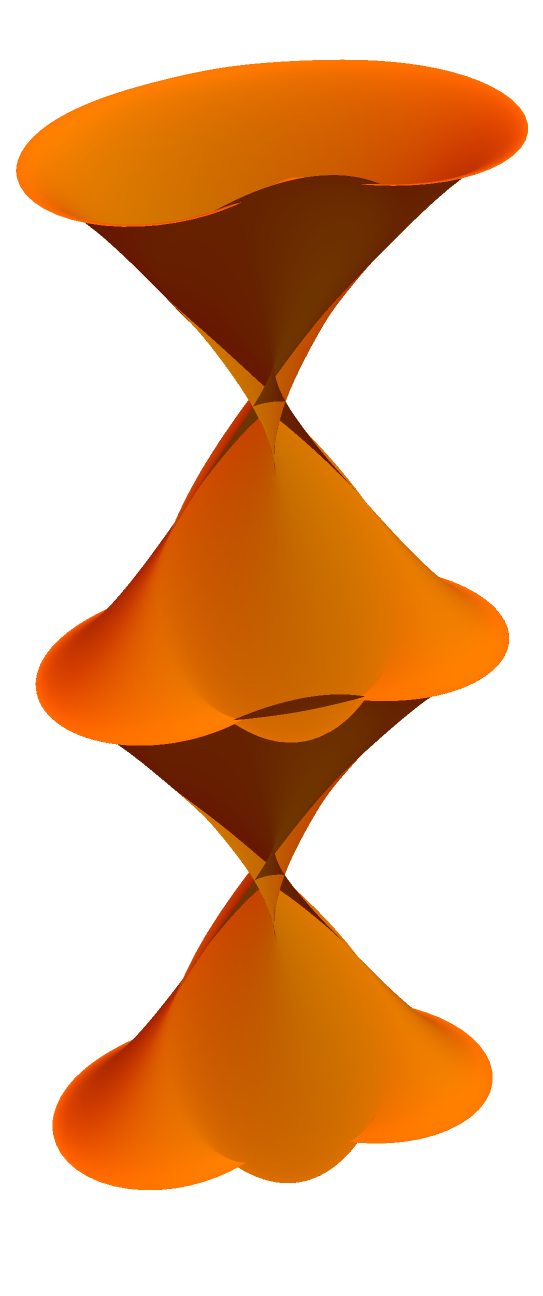}
\includegraphics[scale=0.15]{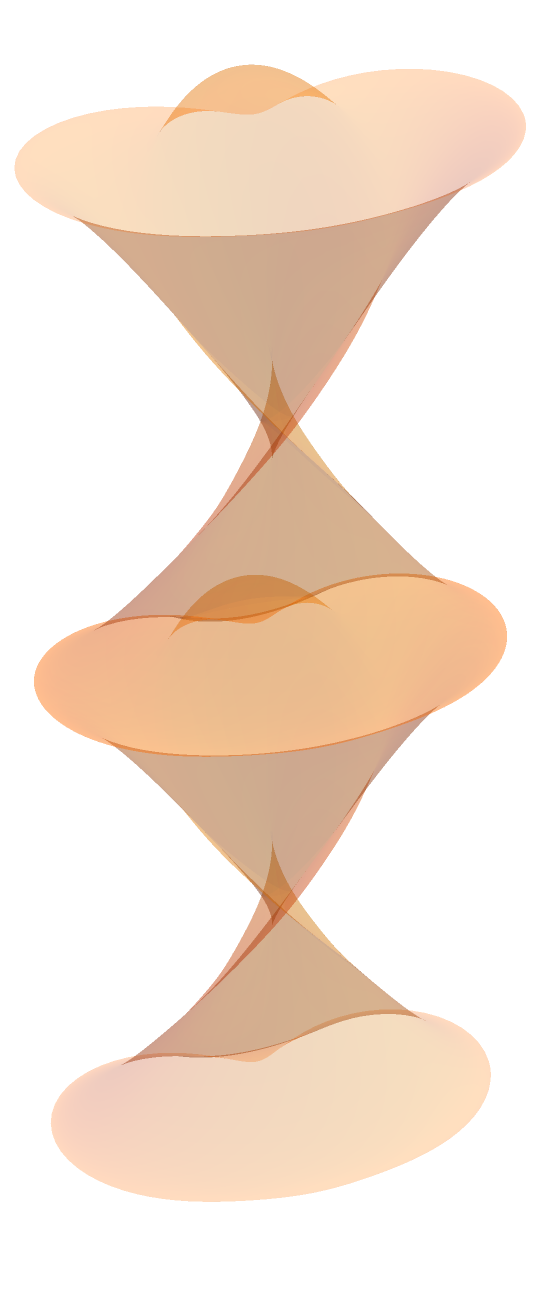}
\includegraphics[scale=0.15]{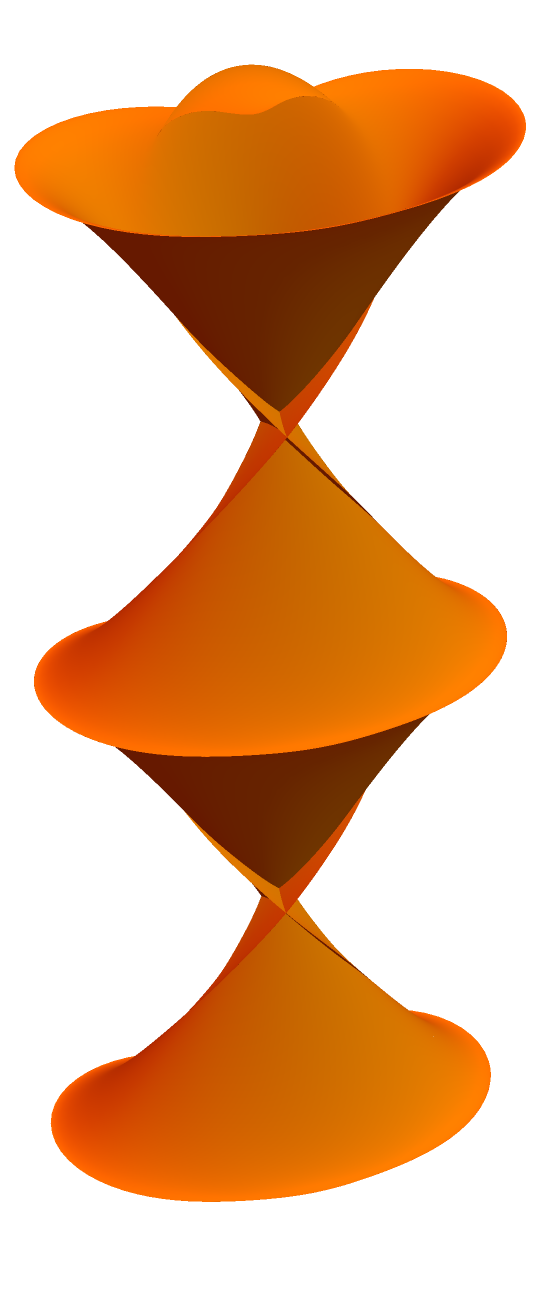}
\caption{An improper affine sphere (with different opacities) generated from a curve in Fig. \ref{FigIntroIASProjection}}
\label{FigIntroIASBeans}
\end{figure}

 In \cite{CPB} (see also \cite{BO, BW, BB}) the dynamics of a quantum particle in the optical lattice potential was investigated. The authors analyze evolution of the Wigner function. The function undergoes a number of catastrophic changes. For a semiclassical approximation the Wigner caustic  consists of the rainbow diagram (the original curve $M$) and a locus of midpoints of chords joining points on rainbow diagram  with parallel tangent lines. But the catastrophe set of exact Wigner function, in addition, contains a locus of midpoints of chords joining points on neighboring rainbow diagrams with parallel tangents. Hence the Wigner caustic of the curve $M$ should be investigated not only locally but globally too.  It turn out that its global geometry is very important for understanding of the quantum-classical correspondence breakdown. It allows to extract important information without using simplifying approximations.  
   
 Singularities of the Wigner caustic for ovals occur exactly from an antipodal pair (the tangent lines at the two points are parallel and the curvatures are equal). The well-known Blaschke-S\"uss theorem states that there are at least three pairs of antipodal points on an oval (\cite{G1, L1}).  The absolute value of the oriented area of the Wigner caustic gives the exact relation between the perimeter and the area of the region bounded by closed regular curves of constant width and improves the classical isoperimetric inequality for convex curves (\cite{Zhang1, Zhang2, Z2, Z3, Z4}). Furthermore this oriented area improves the isoperimetric defect in the reverse isoperimetric inequality (\cite{CGR-Hurwitz-ineq}). Recently the properties of the middle hedgehog, which is a generalization of the Wigner caustic for non-smooth convex curves,  were studied in \cite{S2, S3}. The Wigner caustic in the literature regarding hedgehogs is known also as a projective hedgehog (see \cite{MMY1, MMY2} and the literature therein).  The Wigner caustic could be generalized to obtain an affine $\lambda$-equidistant, which is the locus of points of the above chords which divide the chord segments between base points with a fixed ratio $\lambda$. The singular points of affine equidistants create the Centre Symmetry Set, the natural generalization of the center of symmetry, which is widely studied in \cite{DR1, GH1, GR1, GZ1, J1}. The geometry of an affine extended wave front, i.e. the set $\displaystyle\bigcup_{\lambda\in[0,1]}\{\lambda\}\times E_{\lambda}(M)$, where $E_{\lambda}(M)$ is an affine $\lambda$-equidistant of a~manifold $M$, was studied in \cite{DR1, DZ-gaussbonnet}.

Local properties of singularities of the Wigner caustic and affine equidistants were studied in many papers \cite{CDR1, DJRR1, DMR1, DR1, DRR1, GWZ1, JJR1, OH}. In this paper we study global properties of the Wigner caustic of a generic planar closed curve. In \cite{B1} Berry proved that if $\M$ is a convex curve, then generically the Wigner caustic is a parametrized connected curve with an odd number of cusp singularities and this number is not smaller than $3$. It is not true in general for any closed planar curve. If $\M$ is a~parametrized closed curve with self-intersections or inflexion points then the Wigner caustic has at least two branches (smoothly parametrized components). We present a decomposition of a curve into parallel arcs and thanks to this decomposition we are able to describe the geometry of branches of the Wigner caustic. In general the geometry of the Wigner caustic of a regular closed curve is quite complicated (see Fig. \ref{FigIntro}).

\begin{figure}[h]
\centering
\includegraphics[scale=0.3]{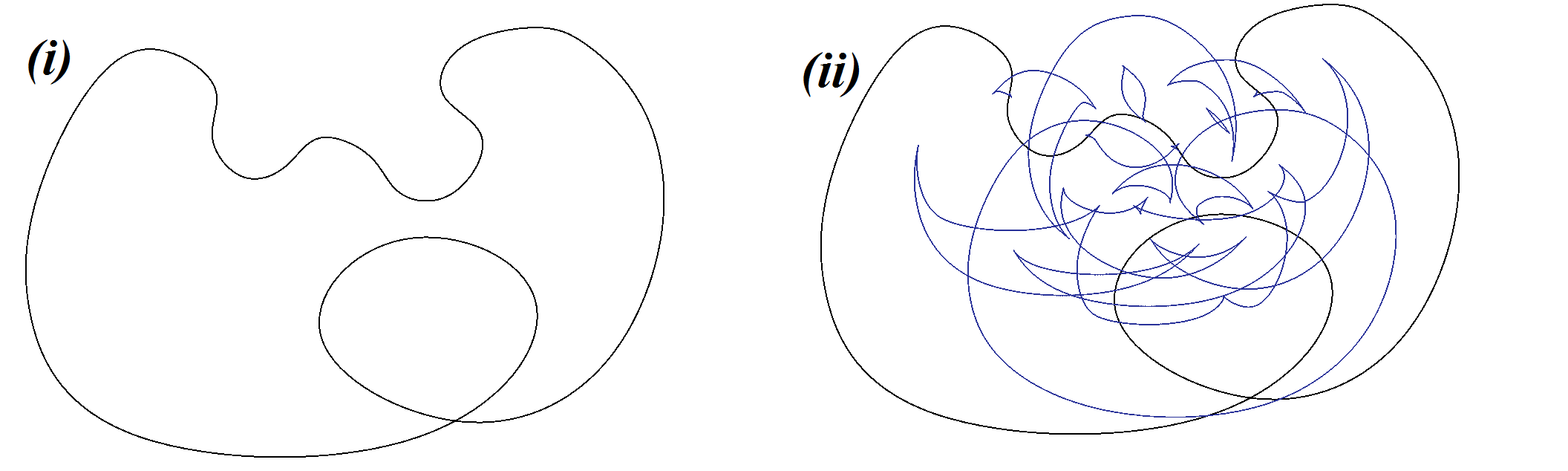}
\caption{(i) A closed regular curve $\M$, (ii) $\M$ and $\Eq_{\frac{1}{2}}(\M)$}
\label{FigIntro}
\end{figure}

In Section \ref{SecGeneralProperties} we briefly sketch the known results on the Wigner caustic and affine equidistants. 

Section \ref{SectionAlgorithm} contains the algorithm to describe branches of the Wigner caustic and affine equidistants of any generic regular parameterized closed curve. Subsection \ref{SectionExample} provides an example of an application of this algorithm to a particular curve.

In the beginning of Section \ref{SectionOnShellGeometryOfConvex} we present global propositions on the number of cusps and inflexion points of the Wigner caustic. We show that the procedure based on a decomposition presented in Section \ref{SectionAlgorithm} can be applied to obtain the number of branches of the Wigner caustic, the number of inflexion points and the parity of the number of cusp singularities of each branch. After that we study global properties of the Wigner caustic on shell, i.e. the branch of the Wigner caustic which connects two inflexion points of a curve. We  present the results on the parity of the number of cusp points of the branches of the Wigner caustic on shell. We also prove that each such branch has even number of inflexion points and there are even number of inflexion points on a path of the original curve between the endpoints of this branch. 

In Section \ref{SectionRosettes} we use the decomposition introduced in Section \ref{SectionAlgorithm} to study the geometry of the Wigner caustic of generic regular closed parameterized curves with non-vanishing curvature and of some generic regular closed parameterized curves with two inflexion points.

All the pictures of the Wigner caustic in this manuscript were made in the application created by the second author \cite{ZwMsc} and in Mathematica \cite{mathematica}.


\section{Preliminaries}\label{SecGeneralProperties}

Let $\M$ be a smooth parameterized curve on the affine plane $\mathbb{R}^2$, i.e. the image of the $C^{\infty}$ smooth map from an interval to $\mathbb{R}^2$. A smooth curve is \textit{closed} if it is the image of a $C^{\infty}$ smooth map from $S^1$ to $\mathbb{R}^2$. A~smooth curve is \textit{regular} if its velocity does not vanish. A regular curve is \textit{simple} if it has no self-intersection points. A~regular simple closed curve is \textit{convex} if its signed curvature has a constant sign. Let $(s_1,s_2)\ni s\mapsto f(s)\in\mathbb{R}^2$ be a parameterization of $\M$. A point $f(s_0)$ is a \textit{$C^k$ regular point} of $\M$ for $k=1, 2, \ldots$ or $k=\infty$ if there exists $\varepsilon>0$ such that $f\big((s_0-\varepsilon, s_0+\varepsilon)\big)$ is a $C^k$ smooth $1$-dimensional manifold. A point $f(s_0)$ is a \textit{singular point} if it is not $C^k$ regular for any $k>0$. A curve is \textit{singular} if it has at least one singular point. A singular point $p$ is called a \textit{cusp} if $\M$ is locally diffeomorphic at $p$ to a curve $(-1,1)\ni t\mapsto (t^2,t^3)\in\mathbb{R}^2$ at $0$. A point $f(s_0)$ is a~cusp of $M$ if and only if $f'(s_0)=0$ and the vectors $f''(s_0)$ and $f'''(s_0)$ are linearly independent. A point $f(s_0)$ is an \textit{inflexion point} of $\M$ if its signed curvature changes sign. An inflexion point $f(s_0)$ is \textit{non-degenerate} (or \textit{ordinary}) if $\det\big(f'(s_0),f'''(s_0)\big)\neq 0$ which means that the order of contact of $\M$ with the tangent line to $\M$ at $f(s_0)$ is equal to $2$. If the curvature vanishes at $f(s_0)$ but does not change sign, then this point is called an \textit{undulation point}.

\begin{rem}
Let $(a,b)\ni s\mapsto f(s)\in\mathbb{R}^2$ be a parameterization of $\M$, then the order of contact of $\M$ with the tangent line to $\M$ at $p=f(t)$ is $k$ if and only if
\begin{align}
\det\left[\frac{\mathrm{d}^if}{\mathrm{d}s^i}(t), \frac{\mathrm{d}f}{\mathrm{d}s}(t)\right]=0\ \text{for }i=1, 2, \ldots, k\text{ and }\det\left[\frac{\mathrm{d}^{k+1}f}{\mathrm{d}s^{k+1}}(t), \frac{\mathrm{d}f}{\mathrm{d}s}(t)\right]\neq 0.
\end{align}
\end{rem}

\begin{defn}\label{parallelpair}
A pair of points $a,b\in\M$ ($a\neq b$) is called a \textit{parallel pair} if the tangent lines to $\M$ at $a$ and $b$ are parallel.
\end{defn}

\begin{defn}\label{chord}
A \textit{chord} passing through a pair $a,b\in\M$, is the line:
$$l(a,b)=\left\{\lambda a+(1-\lambda)b\ \big|\ \lambda\in\mathbb{R}\right\}.$$
\end{defn}

Let $A$ be a subset of $\mathbb{R}^2$, then $\mbox{cl}{A}$ denotes the closure of $A$.

\begin{defn}\label{equidistantSet}
The \textit{Wigner caustic} of $\M$ is the following set
$$\Eq_{\frac{1}{2}}(\M)=\textrm{cl}{\left\{\frac{a+b}{2}\ \big|\ a,b \text{ is a parallel pair of } \M\right\}}.$$
\end{defn}

\begin{rem}
The Wigner caustic of $\M$ is an example of an affine $\lambda$-equidistant set of $\M$,
$$\Eq_\lambda(\M)=\textrm{cl}\left\{\lambda a+(1-\lambda)b\ \big|\ a,b\text{ is a parallel pair of }\M\right\},$$ where $\lambda=\frac{1}{2}$.
Definition \ref{equidistantSet} is different from definitions in papers \cite{CDR1, DMR1, DR1, DRR1, GR1, GWZ1, JJR1, Z1, Z2, Z3}. The closure in  the definition is needed to include inflexion points of $\M$ in $\Eq_{\lambda}(\M)$. For details see Remark \ref{RemInflPoint}.
\end{rem}

Note that, for any given $\lambda\in\mathbb{R}$, we have $\Eq_{\lambda}(\M)=\Eq_{1-\lambda}(\M)$. Thus, the case $\lambda=\frac{1}{2}$ is special. In particular we have $\Eq_0(\M)=\Eq_1(\M)=\M$ if $\M$ is closed.

\begin{defn}
The \textit{Centre Symmetry Set} of $\M$, denoted by $\Css(\M)$, is the envelope of all chords passing through parallel pairs of $\M$.
\end{defn}

Bitangent lines of $\M$ are parts of $\Css(\M)$ (\cite{DR1, GH1}). If $\M$ is a generic convex curve, then $\Css(\M)$, the Wigner caustic and $\Eq_{\lambda}(\M)$ for a generic $\lambda$ are smooth closed curves with at most cusp singularities (\cite{B1, GH1, GZ1, J1}), cusp singularities of all $\Eq_{\lambda}(\M)$ are on regular parts of $\Css(\M)$ (\cite{GZ1}), the number of cusps of $\Css(\M)$ and $\Eq_{\frac{1}{2}}(\M)$ is odd and not smaller than $3$ (\cite{B1, GH1}, see also \cite{G1}), the number of cusps of $\Css(\M)$ is not smaller than the number of cusps of $\Eq_{\frac{1}{2}}(\M)$ (\cite{DR1}).

Let us denote by $\kappa_{\M}(p)$ the signed curvature of a smooth regular curve $\M$ at $p$. Let $a, b$ be a parallel pair of $\M$. Assume that $\kappa_{\M}(b)\neq 0$. Let us fixed local arc length parameterizations of $\M$ nearby the points $a,b$ by $f:(s_0,s_1)\to\mathbb{R}_2$ and by $g:(t_0,t_1)\to\mathbb{R}^2$, respectively. Let's assume that the parameterizations at $a$ and $b$ are in opposite directions, i.e. the velocities at $a,b$ are opposite. Then there exists a function $t:(s_0,s_1)\to(t_0,t_1)$ such that
\begin{align}
f'(s)=-g'(t(s)).
\end{align}
It is easy to see that by the implicit function theorem the function $t$ is smooth and 
\begin{align}
t'(s)=\dfrac{\kappa_{\M}(f(s))}{\kappa_{\M}(g(t(s)))}.
\end{align}
Then by 
\begin{align}
\label{eq_WCNatural}\gamma_{\Eq_{\frac{1}{2}}}(s)=\dfrac{1}{2}\big(f(s)+g(t(s))\big)
\end{align}
we will denote a \textit{local natural parameterization of the Wigner caustic}. Whenever we will write about singular points of the Wigner caustic we will denote these points as the singular points of the parameterization given by \eqref{eq_WCNatural}.

 By direct calculations we get the following lemma.

\begin{lem}\label{LemParallelCurvature}
Let $\M$ be a regular curve. Let $a,b$ be a parallel pair of $\M$, such that $\M$ is parameterized at $a$ and $b$ in opposite directions and $\kappa_{\M}(b)\neq 0$. Let $p=\frac{a+b}{2}$ be a regular point of $\Eq_{\frac{1}{2}}(\M)$.
Then 
\begin{enumerate}[(i)]
\item the tangent line to $\Eq_{\frac{1}{2}}(\M)$ at $p$ is parallel to the tangent lines to $\M$ at $a$ and $b$.

\item the curvature of $\Eq_{\frac{1}{2}}(\M)$ at $p$ is equal to

$$\kappa_{\Eq_{\frac{1}{2}}(\M)}(p)=\frac{2\kappa_{\M}(a)|\kappa_{\M}(b)|}{\left|\kappa_{\M}(b)-\kappa_{\M}(a)\right|}.$$
\end{enumerate}
\end{lem}

Lemma \ref{LemParallelCurvature}(ii) implies the following propositions.

\begin{prop}\label{PropSingularPointOfEq} \cite{GZ1}
Let $a,b$ be a parallel pair of a regular curve $\M$, such that $\M$ is parameterized at $a$ and $b$ in opposite directions and one of $a$ and $b$ is not an inflexion point. Then the point $\frac{a+b}{2}$ is a singular point of $\Eq_{\frac{1}{2}}(\M)$ if and only if $\kappa_{\M}(a)=\kappa_{\M}(b)$.
\end{prop}

\begin{prop}\label{PropInflOfEq}
Let $a, b$ be a parallel pair of a regular closed curve $\M$. Then $\frac{a+b}{2}$ is an inflexion point of $\Eq_{\frac{1}{2}}(\M)$ if and only if one of the points $a$, $b$ is an inflexion point of $\M$.
\end{prop}

Let $\tau_{p}$ denote the translation by a vector $p\in\mathbb{R}^2$.

\begin{defn}
A curve $\M$ is \textit{curved in the same side at $a$ and $b$} (resp. \textit{curved in the different sides}), where $a, b$ is a parallel pair of $\M$, if the center of curvature of $\M$ at $a$ and the center of curvature of  $\tau_{a-b}(\M)$ at $a=\tau_{a-b}(b)$ lie on the same side (resp. on the different sides) of the tangent line to $\M$ at $a$.
\end{defn}

We illustrate above definition in Fig. \ref{FigCurved}.

\begin{figure}[h]
\centering
\includegraphics[scale=0.5]{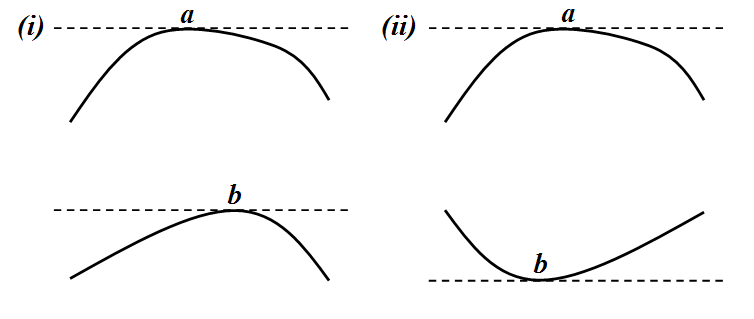}
\caption{(i) A curve curved in the same side at a parallel pair $a,b$, (ii) a curve curved in the different sides at a parallel pair $a,b$}
\label{FigCurved}
\end{figure}

\begin{cor}\label{PropRegularPointOfEq}
If $\M$ is curved in the same side at a parallel pair $a,b$, then $\frac{a+b}{2}$ is a regular point of the Wigner caustic of $\M$.
\end{cor}
\begin{proof}
Let us locally parameterize $\M$ at $a$ and $b$ in opposite directions. Then by Proposition \ref{PropSingularPointOfEq} a point $\frac{a+b}{2}$ is a singular point of $\Eq_{\frac{1}{2}}(\M)$ if and only if 
\begin{align}\label{FormulaPropRegularPointOfEq}
\displaystyle\frac{\kappa_{\M}(a)}{\kappa_{\M}(b)}=1.
\end{align}
The right hand side of (\ref{FormulaPropRegularPointOfEq}) is positive, then $\kappa_{\M}(a)$ and $\kappa_{\M}(b)$ have the same sign, therefore $\M$ is curved in the different sides at $a$ and $b$.

\end{proof}

We denote by $C^{\infty}(S^1,\mathbb{R}^2)$ the set of $C^{\infty}$ mappings from $S^1$ to $\mathbb{R}^2$, i.e. the set of smooth closed parameterized planar curves, and by $\cdot$ the dot product in $\mathbb{R}^2$.

\begin{rem}\label{RemSingCuspWC}
Let $f(s_1),f(s_2)$ be a parallel pair of $f\in C^{\infty}(S^1,\mathbb{R}^2)$. A singular point $\frac{f(s_1)+f(s_2)}{2}$ of the Wigner caustic of $f$ (that is a point for which $\kappa_f(s_1)=\kappa_f(s_2)$ and $f'(s_1)\cdot f'(s_2)<0$) is a cusp if and only if $\kappa_f'(s_1)\neq\kappa_f'(s_1)$, where $\kappa_f$ denotes the signed curvature of $f$ with respect to the parameterization of $f$, and $\kappa_f'$ denotes the derivative of the curvature with the respect to the arc length parameter (\cite{DMR1}).
\end{rem}

\begin{thm}
Let $\mathcal{G}$ be the subset of $C^{\infty}(S^1,\mathbb{R}^2)$ such that each curve $f$ in $\mathcal{G}$ satisfies the following conditions:
\begin{enumerate}[(i)]
\item $f$ is a regular curve with only non-degenerate inflexion points and no undulation points,
\item $f$ has only transversal self crossings,
\item if $f(s_1),f(s_2)$ is a parallel pair of $f$, then $f(s_1)$ or $f(s_2)$ is not an inflexion point,
\item if $f(s_1),f(s_2)$ is a parallel pair of $f$, the points $f(s_1),f(s_2)$ are not inflexion points of $f$, the dot product $f'(s_1)\cdot f'(s_2)$ is negative, and $\kappa_f(s_1)=\kappa_f(s_2)$, then $\kappa_f'(s_1)\neq \kappa_f'(s_1)$, where $\kappa'$ denote the derivative of the curvature with respect to the arc length parameter.
\end{enumerate}
Then $\mathcal{G}$ is a generic subset of $C^{\infty}(S^1,\mathbb{R}^2)$ with Whitney $C^{\infty}$ topology and the Wigner caustic of $f\in\mathcal{G}$ is the finite union of smooth curves with at most cusp singularities.
\end{thm}
\begin{proof} 

Since the intersection of two generic subsets is still a generic subset, it is enough to show that properties from each point are generic. The set of smooth regular closed curves is an open and dense subset of $C^{\infty}(S^1,\mathbb{R}^2)$ because the set of $1$-jets of smooth non-regular closed curves is a smooth submanifold of $J^1(S^1,\mathbb{R}^2)$ of codimension $2$. Let $f:S^1\to\mathbb{R}^2$ be smooth and regular. Having only non-degenerate inflexion points and no undulation points is equivalent to the following property:
\begin{align}
\label{GenOneCond}\det(f'(s),f''(s))=0\ \Rightarrow\ \det(f'(s),f'''(s))\neq 0.
\end{align}

Condition \eqref{GenOneCond} means that the map $j^3f:S^1\to J^3(S^1,\mathbb{R}^2)$ is transversal to the following submanifold of $J^3(S^1,\mathbb{R}^2)$:
$$\left\{j^3g(s)\in J^3(S^1,\mathbb{R}^2)\, \big|\, g'(s)\neq 0, \det\big(g'(s),g''(s)\big)=0\right\}.$$
By the Thom Transversality Theorem (e.g. see Theorem 4.9 in \cite{GGBook}) Property (i) is generic.

To prove generecity of the conditions (ii-iv) we will use the Thom Transversality Theorem for multijets (e.g. see Theorem 4.13 in \cite{GGBook} for details). We denote by $J^k_s(S^1,\mathbb{R}^2)$ the \textit{$s$-fold $k$-jet bundle} and by $(S^1)^{(2)}$ the set $\left(S^1\times S^1\right)\setminus\{(s,s)\ |\ s\in S^1\}$.

Genericity of (ii) follows from transversality of $j^1_2f:(S^1)^{(2)}\to J^1_2(S^1,\mathbb{R}^2)$
to the following submanifold of $J^1_2(S^1,\mathbb{R}^2)$:
$$\left\{(j^1g(s_1),j^1h(s_2))\in J^1_2(S^1,\mathbb{R}^2)\,\big|\, g(s_1)=h(s_2), g'(s_1)\neq 0, h'(s_2)\neq 0\right\}.$$
Transversality means that if $f(s_1)=f(s_2)$ for $s_1\neq s_2$, then $\det(f'(s_1),f'(s_2))\neq 0$. Therefore, Condition (ii) is generic.

Generecity of Property (iii) follows from transversality of the second multijet $j^2_2f:(S^1)^{(2)}\to J^2_2(S^1,\mathbb{R}^2)$ to the submanifold 
$$\left\{(j^2g, j^2h)\in J^2_2(S^1,\mathbb{R}^2)\,\big|\, \det\big(g'(s_1),h'(s_2)\big)=0, g'(s_1)\neq 0, h'(s_2)\neq 0\right\}.$$
This means that if $\det\big(f'(s_1),f'(s_2)\big)=0$ for $s_1\neq s_2$, then $\kappa_f^2(s_1)+\kappa_f^2(s_2)\neq 0$. Hence, Property (iii) is generic.

Now we assume that $f$ satisfies (iii). Genericity of Property (iv) follows from the transversality of  $j^3_2f:(S^1)^{(2)}\to J^3_2(S^1,\mathbb{R}^2)$ to the submanifold 
\begin{align*}
W:=\big\{(j^3g,j^3h)\in 
&
J^3_2(S^1,\mathbb{R}^2)\,\big|\, g'(s_1)\neq 0, h'(s_2)\neq 0, \det(g'(s_1), h'(s_2))=0,\\ &
g'(s_1)\cdot h'(s_2)<0, \kappa_g(s_1)=\kappa_h(s_2)\big\}.
\end{align*}
By direct calculations one can show that this means that if $j^3_2f(s_1,s_2)\in W$, then $\kappa'_f(s_1)\neq\kappa_f'(s_2)$, which is equivalent to the condition for cusp singularity in a singular point of the Wigner caustic (see Remark \ref{RemSingCuspWC}).

\end{proof}

From now one, when we will talk about \textit{generic curves}, we will mean a curve from the set $\mathcal{G}$. Furthermore, generecity of $f$ implies the following geometric properties of $f$.

\begin{prop}\cite{DRZ1}
If $f\in C^{\infty}(S^1, \mathbb{R}^2)$ has only non-degenerate inflexion points and has no undulation points, then the number of inflexion points of $f$ and the rotation number of $f$ are finite.
\end{prop}
 
\begin{defn}
The \textit{tangent line of} $\Eq_{\frac{1}{2}}(\M)$ \textit{at a cusp point} $p$ is the limit of a sequence of $1$-dimensional vector spaces $T_{q_n}\M$ in $\mathbb{R}P^1$ for any sequence $q_n$ of regular points of $\Eq_{\frac{1}{2}}(\M)$ converging to $p$.
\end{defn}

This definition does not depend on the choice of a converging sequence of regular points. By Lemma \ref{LemParallelCurvature}(i) we can see that the tangent line to $\Eq_{\frac{1}{2}}(\M)$ at the cusp point $\frac{a+b}{2}$ is parallel to tangent lines to $\M$ at $a$ and $b$.

\begin{rem}
If $\M$ is an oval, then we have well defined the continuous normal vector field on the double covering $\M$ of $\Eq_{\frac{1}{2}}(\M)$, $\M\ni a\mapsto\frac{a+b}{2}\in\Eq_{\frac{1}{2}}(\M)$ by taking the normal vector to $\M$ compatible with the parameterization of $M$ at the point $a$, and defining this vector as a normal vector to the Wigner caustic at $\frac{a+b}{2}$.
\end{rem}

Let us notice that the continuous normal vector field to $\Eq_{\frac{1}{2}}(\M)$ at regular and cusp points is perpendicular to the tangent line to $\Eq_{\frac{1}{2}}(\M)$. Using this fact and the above definition we define the rotation number in the following way.

\begin{defn}\label{DefRotationNumber}
The \textit{rotation number} of the Wigner caustic of a generic curve $\M$ is the rotation number of the continuous normal vector field of the Wigner caustic.
\end{defn}

Moreover by Lemma \ref{LemParallelCurvature}(i) we can easily get the next two propositions.

\begin{rem}\label{RemInflPoint}
Let $p$ be a inflexion point of $\M$. Then $\Css(\M)$ is tangent to this inflexion point and has an endpoint there. The set $\Eq_{\lambda}(\M)$ for $\lambda\neq \frac{1}{2}$ has an inflexion point at $p$ (as the limit point) and is tangent to $\M$ at $p$. The Wigner caustic is tangent to $\M$ at $p$ too and it has an endpoint there. The Wigner caustic and the Centre Symmetry Set approach $p$ from opposite sides (\cite{B1, DMR1, GH1, GWZ1}). This branch of the Wigner caustic is studied in Section \ref{SectionOnShellGeometryOfConvex}.
\end{rem}

If $\M$ is a generic regular closed curve then $\Eq_{\frac{1}{2}}(\M)$ is a union of smooth parametrized curves. Each of these curves we will call a \textit{smooth branch} of the Wigner caustic of $\M$. In Fig. \ref{PictureAlgExam} we illustrate a non-convex curve $\M$, $\Eq_{\frac{1}{2}}(\M)$, and different smooth branches of $\Eq_{\frac{1}{2}}(\M)$.

\begin{figure}[h]
\centering
\includegraphics[scale=0.27]{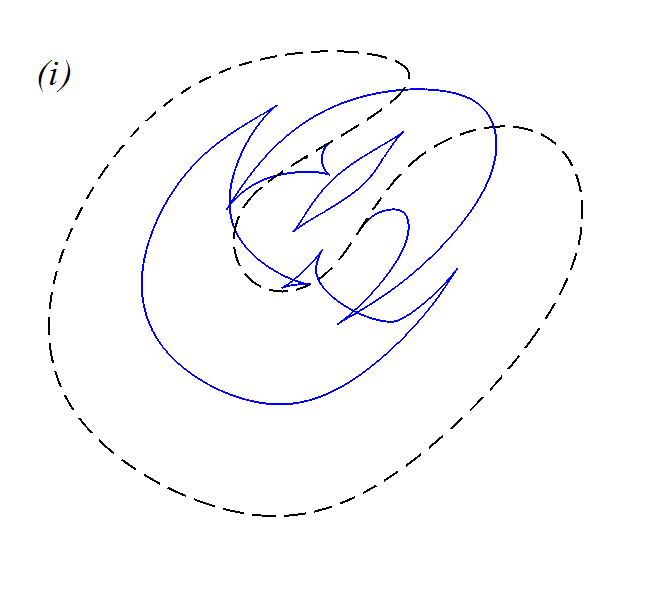}
\includegraphics[scale=0.27]{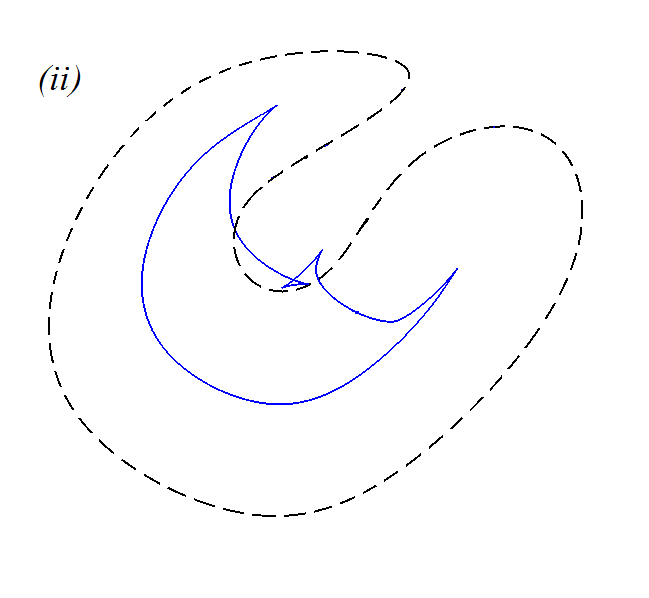}
\vspace{0.25cm}
\includegraphics[scale=0.27]{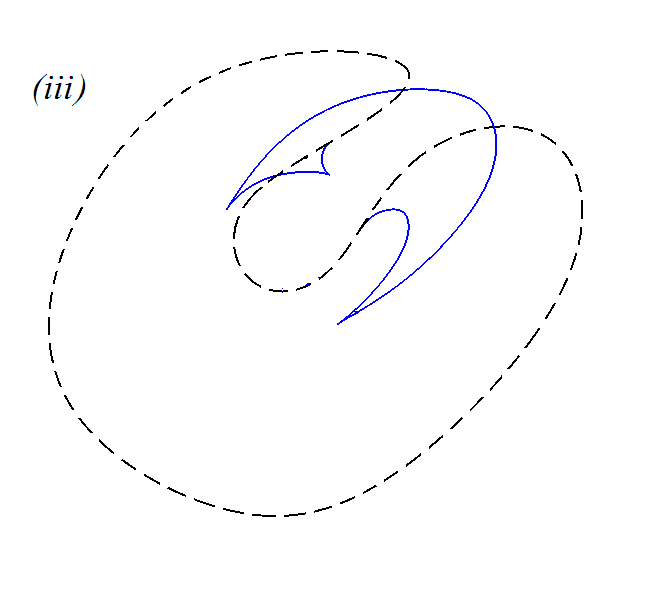}
\includegraphics[scale=0.27]{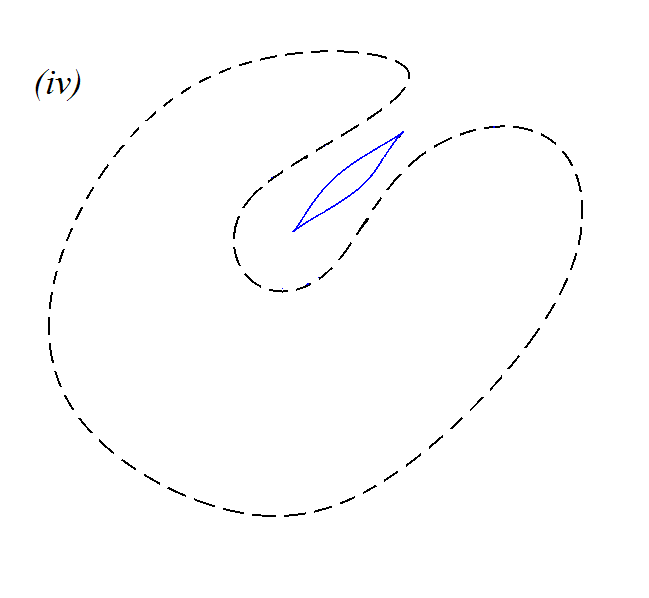}
\caption{(i) A non-convex curve $\M$ with two inflexion points (the dashed line) and $\Eq_{\frac{1}{2}}(\M)$,
 (ii-vi) $\M$ and different smooth branches of $\Eq_{\frac{1}{2}}(\M)$}
\label{PictureAlgExam}
\end{figure}

\section{A Decomposition of a Curve into Parallel Arcs}\label{SectionAlgorithm}

In this section we assume that $\M$ is a generic regular closed curve. We will present a decomposition of $M$ infto parallel arcs which will help us to study the geometry of the smooth branches of the Wigner caustic of $\M$.

\begin{defn}\label{DefAngleFunctin}
Let $S^1\ni s\mapsto f(s)\in\mathbb{R}^2$ be a parameterization of a smooth closed curve $\M$, such that $f(0)$ is not an inflexion point. A function $\varphi_{\M}:S^1\to [0,\pi]$ is called an \textit{angle function of} $\M$ if $\varphi_{\M}(s)$ is the oriented angle between $f'(s)$ and  $f'(0)$ modulo $\pi$. We identify the set $[0,\pi]$ with $S^1$ modulo $\pi$.
\end{defn}

\begin{defn}\label{DefLocalExtrema}
A point $\varphi$ in $S^1$ is a \textit{local extremum }of $\varphi_{\M}$ if there exists $s$ in $S^1$ such that $\varphi_{\M}(s)=\varphi$, $\varphi'_{\M}(s)=0$, $\varphi''_{\M}(s)\neq 0$. The local extremum $\varphi$ of $\varphi_{\M}$ is a \textit{local maximum (resp. minimum)} if $\varphi''_{\M}(s)<0$ (resp. $\varphi''_{\M}>0$). We denote by $\mathcal{M}(\varphi_{\M})$ the set of local extrema of $\varphi_{\M}$.  
\end{defn}

The angle function has the following properties.

\begin{prop}\label{PropAlgAngleFun}
Let $\M$ be a generic regular closed curve. Let $f$ be the arc length parameterization of $\M$ and let $\varphi_{\M}$ be the angle function of $\M$. Then
\begin{enumerate}[(i)]
\item $f(s_1), f(s_2)$ is a parallel pair of $\M$ if and only if $\varphi_{\M}(s_1)=\varphi_{\M}(s_2)$,
\item $\varphi_{\M}'(s)$ is equal to the signed curvature of $\M$ with respect to the parameterization of $\M$,
\item $\M$ has an inflexion point at $f(s_0)$ if and only if $\varphi_{\M}(s_0)$ is a local extremum.
\item if $\varphi_{\M}(s_1), \varphi_{\M}(s_2)$ are local extrema and there is no extremum on $\varphi_{\M}\big((s_1,s_2)\big)$, then one of extrema $\varphi_{\M}(s_1), \varphi_{\M}(s_2)$ is a local maximum and the other one is a local minimum.
\end{enumerate}
\end{prop}

\begin{lem}\label{LemmaAlgEvenNumOfInfl}
Let $\varphi_{\M}$ be the angle function of a generic regular closed curve $\M$. Then the function $\varphi_{\M}$ has an even number of local extrema, i.e. $\M$ has an even number of inflexion points .
\end{lem}
\begin{proof}
It is a consequence of the fact that the number of local extrema of a generic smooth function from $S^1$ to $S^1$ is even.

\end{proof}

Let $\varphi_{\M}$ be the angle function of $\M$. 

\begin{defn}\label{DefSeqOfLocExtr}
The \textit{sequence of local extrema} is the following sequence\linebreak $(\varphi_0, \varphi_1, \ldots, \varphi_{2n-1})$ where $\{\varphi_0, \varphi_1, \ldots, \varphi_{2n-1}\}=\mathcal{M}(\varphi_{\M})$ and the order is compatible with the orientation of $S^1=\varphi_{\M}(S^1)$.
\end{defn} 

\begin{defn}\label{DefSeqOfParallPts}
The \textit{sequence of division points} $\mathcal{S}_{\M}$ is the following sequence\linebreak $(s_0, s_1, \ldots, s_{k-1})$, where $\{s_0, s_1, \ldots, s_{k-1}\}=\varphi_{\M}^{-1}\left(\mathcal{M}(\varphi_{\M})\right)$ if $\mathcal{M}(\varphi_{\M})$ is not empty, otherwise $\{s_0, s_1, \ldots, s_{k-1}\}=\varphi_{\M}^{-1}\left(\varphi_{\M}(0)\right)$, and the order of $\mathcal{S}_{\M}$ is compatible with the orientation of $\M$.   
\end{defn}

In the sequence of division points there are points $s\in S^1$ such that $f(s), p$ is a parallel pair, where $p$ is an inflexion point of $\M$ (if $M$ has not inflexion point, then $p=f(0)$). In Fig. \ref{PictureAlgFindingPhiSets} we illustrate an example of a closed regular curve $\M$, the angle function $\varphi_{\M}$ and the sequence of division points. Let us notice that the images of points in the sequence of division points divide the curve $M$ in arcs. Some of these arcs (say $\mathcal{A}_1$ and $\mathcal{A}_2$) have the property that for any point $a_i\in\mathcal{A}_i$ there exists a point $a_j\in\mathcal{A}_j$ such that $a_i, a_j$ is a parallel pair for $i\neq j\in\{1,2\}$. Such arcs we will call \textit{parallel arcs}. The set of arcs splits into subsets such that any two arcs in the same subset are parallel (see Definition \ref{DefSetParallArcs}).

\begin{figure}[h]
\centering
\includegraphics[scale=0.33]{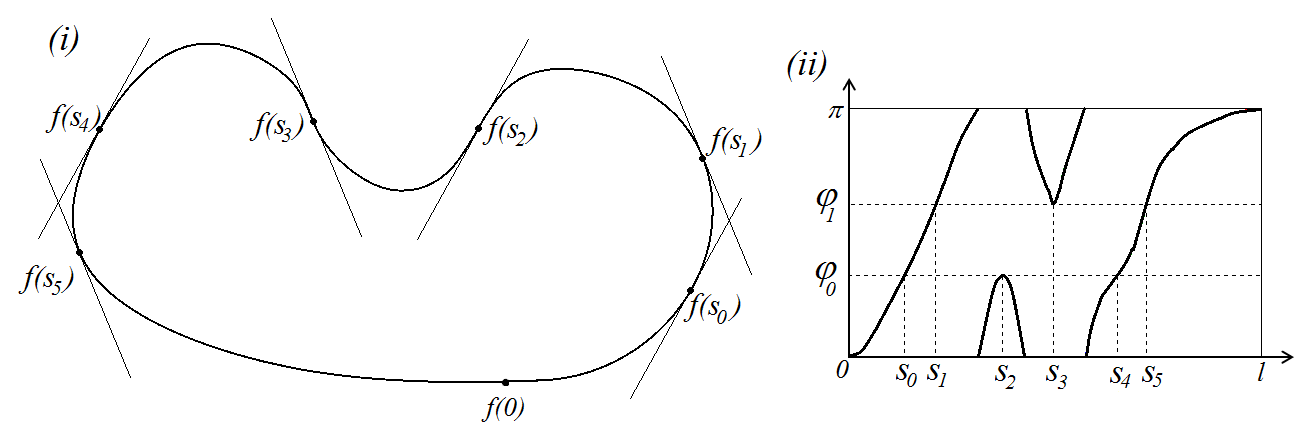}
\caption{(i) A closed regular curve $\M$ with points $f(s_i)$ and tangent lines to $\M$ at these points, (ii) a graph of the angle function $\varphi_{\M}$ with $\varphi_i$ and $s_i$ values}
\label{PictureAlgFindingPhiSets}
\end{figure}

\begin{prop}\label{PropParalllel-infexion}
If $\M$ is a generic regular closed curve and $a\in M$ is an inflexion point then 
the number of points $b\in M$, such that $b\ne a$ and $a$, $b$ is a parallel pair, is even.  
\end{prop}

\begin{proof}
There are no inflexion points $b\in M$ such that $a$, $b$ is a parallel pair and the point $a$ is not a self-intersection point of $M$, since the curve $M$ is generic. The inflexion points of $M$ correspond to local extrema of the angle function  $\varphi_{\M}$.
We divide the graph of the angle function $\varphi_{\M}$ into continuous paths of the form 
\begin{align*}
\big\{(t,\varphi_{M}(t))\ |\ t\in[t_1,t_2], \varphi_{M}(t_1),\varphi_{M}(t_2)\in\{0,\pi\}, \forall t\in(t_1,t_2)\ \varphi_{\M}(t)\notin \{0,\pi\}\big\}.
\end{align*}

Let $\alpha$ belong to $(0,\pi)$. First we assume that $\alpha$ is not equal to a local extremum of a path $\mathcal{P}$.  Then a line $\varphi=\alpha$  intersects the path $\mathcal{P}$ an even number of times if $\mathcal{P}$ is a path from $0$ to $0$ or from $\pi$ to $\pi$, since both the beginning and the end  of  $\mathcal P$ are on the same side of the line (see Fig. \ref{FigSmEven}(i)). This line intersects $\mathcal P$  an odd number of times if $\mathcal{P}$ is a path from $0$ to $\pi$ or from $\pi$ to $0$, since the beginning and the end of $\mathcal P$ are on different sides of the line (see Fig. \ref{FigSmEven}(ii)).

\begin{figure}[h]
\centering
\includegraphics[scale=0.3]{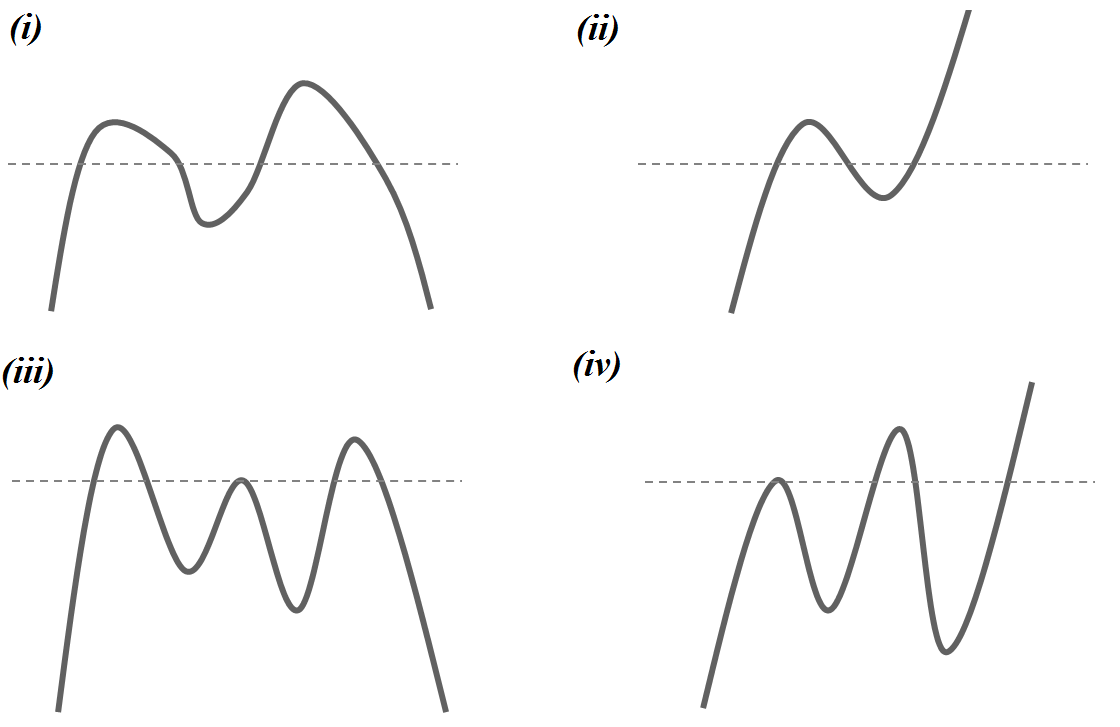}
\caption{Continuous paths}
\label{FigSmEven}
\end{figure}

Now we assume that  $\alpha$ is equal to a local extremum of $\mathcal{P}$. In this case the line $\varphi=\alpha$ intersects a path $\mathcal{P}$ an odd number of times if $\mathcal{P}$ is a path from $0$ to $0$ or from $\pi$ to $\pi$ (see Fig. \ref{FigSmEven}(iii)) and this line intersects $\mathcal P$  an even number of times if $\mathcal{P}$ is a path from $0$ to $\pi$ or from $\pi$ to $0$ (see Fig. \ref{FigSmEven}(iv)), since by a small local vertical perturbation around the extremum point we obtain the previous cases and the numbers of intersection points have a difference $\pm 1$ (see Fig. \ref{FigVertPert}).

\begin{figure}[h]
\centering
\includegraphics[scale=0.3]{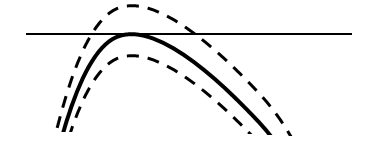}
\caption{Vertical perturbation nearby a local maximum}
\label{FigVertPert}
\end{figure}

Let us notice that a path from $0$ to $0$ or from $\pi$ to $\pi$ corresponds to an arc of a curve with the rotation number equals $0$ and a path from $0$ to $\pi$ or from $\pi$ to $0$ corresponds to an arc of a curve with t he rotation number equals $\pm\frac{1}{2}$. Since the rotation number of $\M$ is an integer, the number of paths from $0$ to $\pi$ or from $\pi$ to $0$ in the graph of $\varphi_{\M}$ is even. Each path of this type intersects every horizontal line $\varphi =\alpha$ at least once. Thus the number of intersections of $\varphi_{\M}$ and the line $\varphi=\varphi_{\M}(f^{-1}(a))$ is odd. But the number of points $b\ne a$ such that $a, b$ is a parallel pair is one less than the number of intersection points of the graph of $\varphi_{\M}$ and the line $\varphi=\varphi_{\M}(f^{-1}(a))$.
\end{proof}

The number of inflexion points of a generic regular closed curve is even.  Thus by Proposition \ref{PropParalllel-infexion} we have the following corollary.

\begin{cor}\label{PropMisEven}
If $\M$ is a generic regular closed curve then $\#\mathcal{S}_M$ is even.
\end{cor}

We recall that in this section we assume that $\M$ is a generic regular closed curve and let $\#\mathcal{S}_{\M}=2m$.

The functions $\minm_{2m}, \maxm_{2m}:\{0,1,\ldots,2m-1\}^2\to\{0,1,\ldots,2m-1\}$ are analogs of the minimum and the maximum functions modulo $2m$, respectively. Namely,
\begin{align*}
\minm_{2m}(k,l):&=\left\{\begin{array}{ll}2m-1, &\text{ if }\{k,l\}=\{0, 2m-1\},\\ \min(k,l), &\text{ otherwise},\end{array}\right.\\
\maxm_{2m}(k,l):&=\left\{\begin{array}{ll}0, &\text{ if }\{k,l\}=\{0, 2m-1\},\\ \max(k,l), &\text{ otherwise}.\end{array}\right.
\end{align*}

We denote by  $(s_{2m-1},s_0)$ an interval $(s_{2m-1}, L_{\M}+s_0)$, where $L_{\M}$ is the length of $M$.

In the following definition indexes $i$ in $\varphi_i$ are computed modulo $2n$, indexes $j, j+1$ in $\overarc{\p_{j}}{\p_{j+1}}$ and $\overarc{\p_{j+1}}{\p_{j}}$ are computed modulo $2m$.

\begin{defn}\label{DefSetParallArcs}
If $\mathcal{M}(\varphi_{\M})=\{\varphi_0, \varphi_1, \ldots, \varphi_{2n-1}\}$, then for every $i\in\{0, 1, \ldots, 2n-1\}$, a \textit{set of parallel arcs} $\Phi_i$ is the following set
\begin{align*}
\Phi_i=\Big\{\overarc{\p_k}{\p_l}\ \big|\ &k-l=\pm 1\ \mbox{mod} (2m),\ \varphi_{\M}(s_k)=\varphi_i,\ \varphi_{\M}(s_l)=\varphi_{i+1}, \\ &\varphi_{\M}\big((s_{\minm_{2m}(k,l)}, s_{\maxm_{2m}(k,l)})\big)=(\varphi_i, \varphi_{i+1})\Big\},
\end{align*}
where $\p_k=f(s_k)$ and $\overarc{\p_k}{\p_l}=f\left(\big[s_{\minm_{2m}(k, l)}, s_{\maxm_{2m}(k, l)}\big]\right)$.

If $\mathcal{M}(\varphi_{\M})$ is empty then we define only one \textit{set of parallel arcs} as follows:
\begin{align*}
\Phi_0=\big\{\overarc{\p_0}{\p_1}, \overarc{\p_1}{\p_2}, \ldots, \overarc{\p_{2m-2}}{\p_{2m-1}}, \overarc{\p_{2m-1}}{\p_0}\big\}.
\end{align*}
\end{defn}

The set of parallel arcs has the following property.

\begin{prop}\label{PropDiffeoBetweenParallelArcs}
Let $f:S^1\to\mathbb{R}^2$ be the arc length parameterization of $\M$.
For every two arcs $\overarc{\p_{k}}{\p_{l}}$, $\overarc{\p_{k'}}{\p_{l'}}$ in $\Phi_i$ the well defined map
\begin{align*}
\overarc{\p_{k}}{\p_{l}}\ni p\mapsto P(p)\in\overarc{\p_{k'}}{\p_{l'}},
\end{align*}
where the pair $p, P(p)$ is a parallel pair of $\M$, is a diffeomorphism.
\end{prop}

\begin{defn}\label{DefGlueingScheme}
Let $\overarc{\p_{k_1}}{\p_{k_2}}$, $\overarc{\p_{k_2}}{\p_{l_2}}$ belong to the same set of parallel arcs, then $\begin{array}{ccc}
\p_{k_1} &\frown&\p_{k_2} \\ \hline
\p_{l_1} &\frown&\p_{l_2} \\ \hline \end{array}$ denotes the following set (the \textit{arc}) 
\begin{align*}
\mbox{cl}\Big\{(a,b)\in\M\times\M\ \Big|\ &a\in\overarc{\p_{k_1}}{\p_{k_2}}, b\in\overarc{\p_{l_1}}{\p_{l_2}},\ a,b\text{ is a parallel pair of }\M\Big\}.
\end{align*}

In addition $\begin{array}{ccccc}
\p_{k_1} &\frown &\ldots &\frown & \p_{k_n}\\ \hline
\p_{l_1} &\frown &\ldots &\frown& \p_{l_n}\\ \hline \end{array}$ denotes $\displaystyle\bigcup_{i=1}^{n-1}\begin{array}{ccc} 
\p_{k_i} &\frown &\p_{k_{i+1}} \\ \hline
\p_{l_i} &\frown &\p_{l_{i+1}} \\ \hline \end{array}$. We will call this set a \textit{glueing scheme}.
\end{defn}

\begin{rem}\label{RemInflexionInScheme}
If $\overarc{\p_k}{\p_l}$ belongs to a set of parallel arcs, then  there are neither inflexion points nor points with parallel tangent lines to tangent lines at inflexion points of $\M$ in $\overarc{\p_k}{\p_l}\setminus\{\p_k, \p_l\}$.
\end{rem}

\begin{defn}
The $\frac{1}{2}$-\textit{point map} (\cite{DR1}) is the map
\begin{align*}
\pi_{\frac{1}{2}}:\M\times\M\to\mathbb{R}^2, (a,b)\mapsto\dfrac{a+b}{2}.
\end{align*}
\end{defn}

Let $\mathcal{A}_1=\overarc{\p_{k_1}}{\p_{k_2}}$ and $\mathcal{A}_2=\overarc{\p_{l_1}}{\p_{l_2}}$ be two arcs of $\M$ which belong to the same set of parallel arcs. It is easy to see that $\Eq_{\frac{1}{2}}\Big(\mathcal{A}_1\cup\mathcal{A}_2\Big)$ consists of one arc $\begin{array}{ccc}
\p_{k_1} &\frown&\p_{k_2} \\ \hline
\p_{l_1} &\frown&\p_{l_2} \\ \hline \end{array}$ under $\pi_{\frac{1}{2}}$ (see Fig. \ref{FigEqFromParallelArcs}). From this observation we get the following proposition.

\begin{figure}[h]
\centering
\includegraphics[scale=0.5]{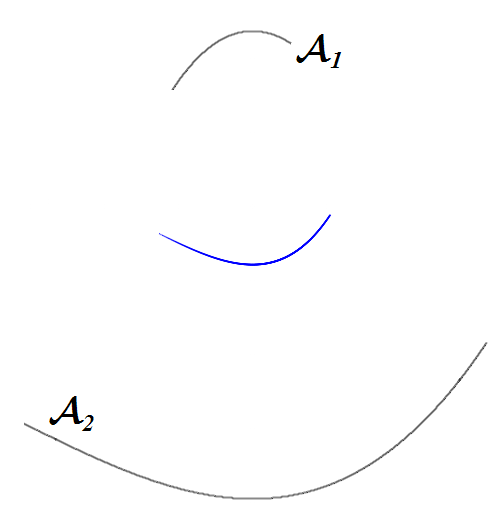}
\caption{Two arcs $\mathcal{A}_1$ and $\mathcal{A}_2$ of $\M$ belonging to the same set of parallel arcs and $\Eq_{\frac{1}{2}}(\mathcal{A}_1\cup\mathcal{A}_2)$}
\label{FigEqFromParallelArcs}
\end{figure}

\begin{prop}\label{PropNumDiffArcs}
The Wigner caustic $\Eq_{\frac{1}{2}}(\M)$ is the image of the union of \linebreak $\displaystyle \sum_{i}{\#\Phi_i\choose 2}$ different arcs under the $\frac{1}{2}$-point map $\pi_{\frac{1}{2}}$.
\end{prop}

\begin{prop}\label{PropAlgAlwaysGoFurhter}
Let $\M$ be a generic regular closed curve which is not convex. If a glueing scheme is of the form $\begin{array}{ccc} 
\p_{k_1} &\frown &\p_{k_2} \\ \hline
\p_{l_1} &\frown &\p_{l_2} \\ \hline \end{array}$, then this scheme can be prolonged in a unique way to $\begin{array}{ccccc} 
\p_{k_1} &\frown &\p_{k_2} &\frown &\p_{k_3}\\ \hline
\p_{l_1} &\frown &\p_{l_2} &\frown &\p_{l_3}\\ \hline \end{array}$ such that $(k_1,l_1)\ne (k_3,l_3)$.
\end{prop}

\begin{proof}

Let us consider
\begin{align}\label{GlueToProlong}
\begin{array}{ccc} 
\p_{k_1} &\frown &\p_{k_2} \\ \hline
\p_{l_1} &\frown &\p_{l_2} \\ \hline \end{array}.
\end{align}
Let $\mathcal{A}_1=\overarc{\p_{k_1}}{\p_{k_2}}\setminus\{\p_{k_1}, \p_{k_2}\}$,  $\mathcal{A}_2=\overarc{\p_{l_1}}{\p_{l_2}}\setminus\{\p_{l_1},\p_{l_2}\}$. By Remark \ref{RemInflexionInScheme} $\mathcal{A}_1$ and $\mathcal{A}_2$ must be curved in the same side or in the opposite sides at any parallel pair in $\mathcal{A}_1\cup\mathcal{A}_2$ (see Fig. \ref{PictureGoFurther}(i--ii)). Let us consider the case in Fig. \ref{PictureGoFurther}(i), the other case is similar. Then (\ref{GlueToProlong}) can be prolonged in the following two ways.
\begin{enumerate}[(1)]
\item Neither $\p_{k_2}$ nor $\p_{l_2}$ is an inflexion point of $\M$. Then (\ref{GlueToProlong}) can be prolonged to 
$\begin{array}{ccccc} 
\p_{k_1} &\frown &\p_{k_2}&\frown &\p_{k_3} \\ \hline
\p_{l_1} &\frown &\p_{l_2}&\frown &\p_{l_3} \\ \hline \end{array}$, where $k_1\neq k_3$ and $l_1\neq l_3$ (see Fig. \ref{PictureGoFurther}(iii)).
\item One of points $\p_{k_2}$, $\p_{l_2}$ is an inflexion point of $\M$. Let us assume that this is $\p_{k_2}$. Then (\ref{GlueToProlong}) can be prolonged to
$\begin{array}{ccccc} 
\p_{k_1} &\frown &\p_{k_2}&\frown &\p_{k_3} \\ \hline
\p_{l_1} &\frown &\p_{l_2}&\frown &\p_{l_1} \\ \hline \end{array}$, where $k_1\neq k_3$ (see Fig. \ref{PictureGoFurther}(iv)).
\end{enumerate}
Since $\M$ is generic, at least one of the points $\p_{k_2}$, $\p_{l_2}$ is not an inflexion point of $\M$.

\begin{figure}[h]
\centering
\includegraphics[scale=0.3]{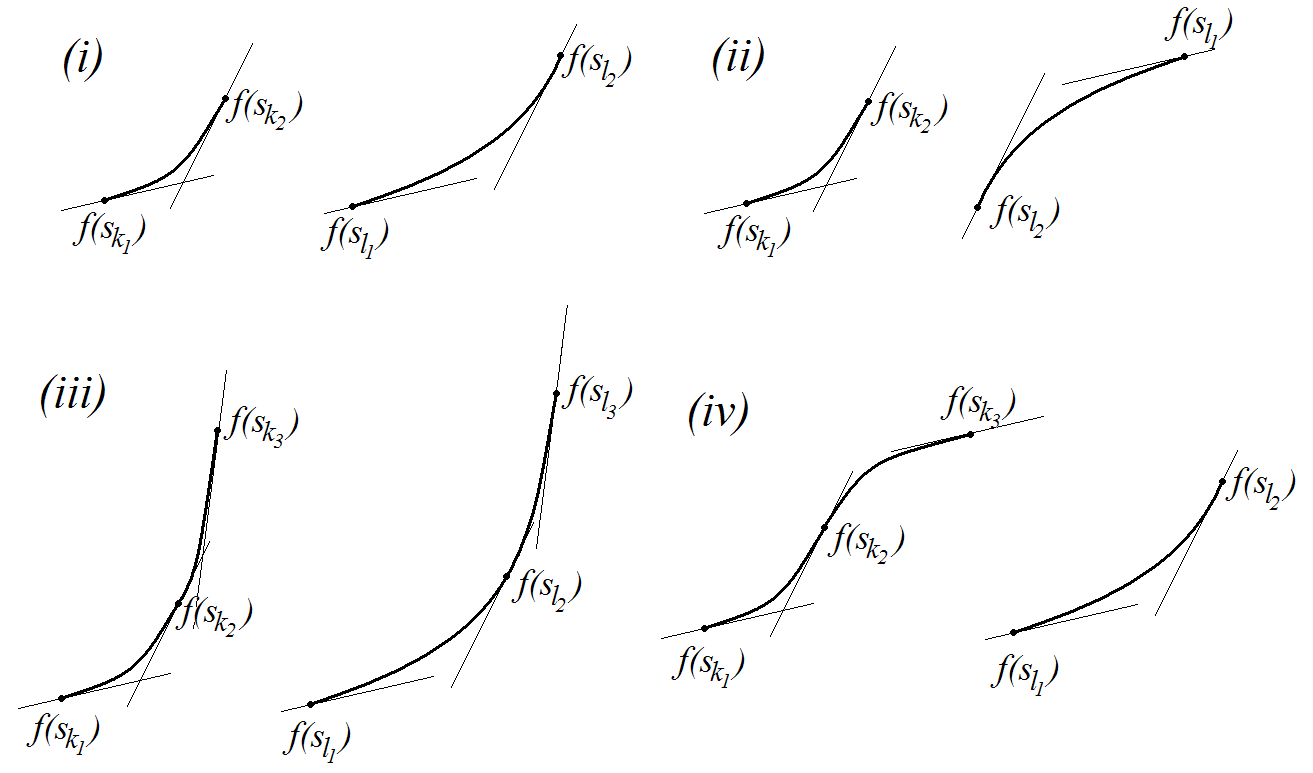}
\caption{Possible prolongations of an arc of a curve}
\label{PictureGoFurther}
\end{figure}

\end{proof}

\begin{rem}
To avoid repetition in the union in Definition \ref{DefGlueingScheme} we assume that no pair $\begin{array}{c} \p_k \\ \hline \p_l \\ \hline \end{array}$ except the beginning and the end can appear twice in the glueing scheme. Furthermore, if the pair $\begin{array}{c}\p_k \\ \hline \p_l \\ \hline \end{array}$ is in the glueing scheme than the pair $\begin{array}{c} \p_l \\ \hline \p_k \\ \hline \end{array}$ does not appear unless they are the beginning and the end of the scheme.
\end{rem}

The image of a glueing scheme under the $\frac{1}{2}$-point map $\pi_{\frac{1}{2}}$ represents parts of branches of the Wigner caustic. If we equip the set of all possible glueing schemes with the inclusion relation, then this set is partially ordered.

There is only finite number of arcs from which we can construct branches of $\Eq_{\frac{1}{2}}(\M)$. Therefore we can define a maximal glueing scheme.

\begin{defn}
A \textit{maximal glueing scheme} is a glueing scheme which is a maximal element of the set of all glueing schemes equipped with the inclusion relation.
\end{defn}

\begin{rem}\label{RemGlueSchemeOval}
If $\M$ is a generic regular convex curve, then the set of parallel arcs is equal to $\Phi_0=\{\overarc{\p_0}{\p_1}, \overarc{\p_1}{\p_0}\}$. Then the only maximal glueing scheme is 
$\begin{array}{ccc}
\p_0&\frown& \p_1\\ \hline
\p_1&\frown& \p_0\\ \hline\end{array}$

\end{rem}

\begin{prop}\label{PropChains}
The set of all glueing schemes equipped with the inclusion relation is the disjoint union of totally ordered sets.
\end{prop}
\begin{proof}
It follows from uniqueness of the prolongation of the glueing scheme (see Proposition \ref{PropAlgAlwaysGoFurhter}).

\end{proof}

\begin{lem}\label{LemPropMaxGlueSchemes}
Let $f:S^1\mapsto\mathbb{R}^2$ be the arc length parameterization of $\M$. Then
\begin{enumerate}[(i)]
\item for every two different arcs $\overarc{\p_{k_1}}{\p_{k_2}}$, $\overarc{\p_{l_1}}{\p_{l_2}}$ in $\Phi_i$ there exists exactly one maximal glueing scheme containing 
$\begin{array}{ccc}
\p_{k_1} &\frown &\p_{k_2} \\ \hline
\p_{l_1} &\frown &\p_{l_2} \\ \hline \end{array}$
or
$\begin{array}{ccc}
\p_{k_2} &\frown &\p_{k_1} \\ \hline
\p_{l_2} &\frown &\p_{l_1} \\ \hline \end{array}$
or
$\begin{array}{ccc}
\p_{l_1} &\frown &\p_{l_2} \\ \hline
\p_{k_1} &\frown &\p_{k_2} \\ \hline \end{array}$
or
$\begin{array}{ccc} 
\p_{l_2} &\frown &\p_{l_1} \\ \hline
\p_{k_2} &\frown &\p_{k_1} \\ \hline \end{array}$.

\item every maximal glueing scheme is in the following form $\begin{array}{ccccc}
\p_k &\frown&\ldots&\frown&\p_{k'} \\ \hline
\p_l &\frown&\ldots&\frown&\p_{l'} \\ \hline \end{array}$, where $\{p_k, p_l\}=\{p_{k'},p_{l'}\}$ whenever $p_k\neq p_l$ and $p_{k'}\neq p_{l'}$.

\item if $\p_k$ is an inflexion point of $\M$, then there exists a maximal glueing scheme which is in the form 
\begin{align*}
\begin{array}{ccccccccc}
\p_k &\frown&\p_{k_1}&\frown&\ldots&\frown&\p_{k_n}&\frown&\p_l \\ \hline
\p_k &\frown&\p_{l_1}&\frown&\ldots&\frown&\p_{l_n}&\frown&\p_l \\ \hline \end{array},
\end{align*}
where $\p_l$ is a different inflexion point of $\M$ and $p_{k_i}\neq p_{l_i}$ for $i=1, 2, \ldots, n$.

\end{enumerate}
\end{lem}
\begin{proof}
(i) is a consequence of the uniqueness of the prolongation of a glueing scheme (see Proposition \ref{PropAlgAlwaysGoFurhter}).

The proof of (ii) follows from (i) and the fact that the following equalities hold: $\begin{array}{ccc}
\p_{k_1} &\frown &\p_{k_2} \\ \hline
\p_{l_1} &\frown &\p_{l_2} \\ \hline \end{array}=\begin{array}{ccc}
\p_{k_2} &\frown &\p_{k_1} \\ \hline
\p_{l_2} &\frown &\p_{l_1} \\ \hline \end{array}$ and $\begin{array}{ccc}
\p_{k_1} &\frown &\p_{k_2} \\ \hline
\p_{l_1} &\frown &\p_{l_2} \\ \hline \end{array}=\begin{array}{ccc}
\p_{l_1} &\frown &\p_{l_2} \\ \hline 
\p_{k_1} &\frown &\p_{k_2} \\ \hline\end{array}$.

To prove (iii) let us prolong $\begin{array}{ccc}
\p_{k} &\frown &\p_{k_1} \\ \hline
\p_{k} &\frown &\p_{l_1} \\ \hline \end{array}$ to the maximal glueing scheme $\mathcal{G}$. Any point $p_l$ in the sequence of division points $\mathcal{S}_{\M}$ belongs to exactly two arcs in all sets of parallel arcs. Then by (ii) this maximal glueing scheme is in the following form
\begin{align}\label{SymGS}
\begin{array}{ccccccccc}
\p_k &\frown&\p_{k_1}&\frown&\ldots&\frown&\p_{k_n}&\frown&\p_l \\ \hline
\p_k &\frown&\p_{l_1}&\frown&\ldots&\frown&\p_{l_n}&\frown&\p_l \\ \hline \end{array},
\end{align}
If \eqref{SymGS} would contain some other inflexion point $\p_r$ in the middle, then \eqref{SymGS} would contain the following part: $$\begin{array}{ccccccccc}
\p_{r'} &\frown &\p_{r}&\frown&\p_{r''} \\ \hline
\p_{r''} &\frown &\p_{r}&\frown&\p_{r'}\\ \hline \end{array}$$
which is impossible by (i).

\end{proof}

\begin{thm}\label{ThmGlueSchemeIsBranch} 
The image of every maximal glueing scheme of $\M$ under the $\frac{1}{2}$-point map $\pi_{\frac{1}{2}}$ is a branch of the Wigner caustic of $\M$ and all branches of the Wigner caustic can be obtain in this way.
\end{thm}
\begin{proof}
Let $f:S^1\to\mathbb{R}^2$ be the arc length parameterization of $\M$.

It is easy to see that 
\begin{align*}
S^1=\bigsqcup_{s\in\mathcal{S}_{\M}}\{s\}\sqcup\bigcup_{i}\bigsqcup_{(k,l)\in\Phi_i} (s_{\minm_{2m}(k,l)},s_{\maxm_{2m}(k,l)})
\end{align*}
and then
\begin{align*}
M=\bigcup_{s\in\mathcal{S}_{\M}}\{f(s)\}\cup\bigcup_{i}\bigcup_{(k,l)\in\Phi_i} f\Big((s_{\minm_{2m}(k,l)},s_{\maxm_{2m}(k,l)})\Big),
\end{align*}
where $\sqcup$ denotes the disjoint union. Then by Proposition \ref{PropDiffeoBetweenParallelArcs} we obtain that

\begin{align}\label{EqUnionOfArcs}
\Eq_{\frac{1}{2}}(\M)=\bigcup_i\bigcup_{\substack{\overarc{\p_{k}}{\p_{l}}, \overarc{\p_{k'}}{\p_{l'}}\in\Phi_i \\ \overarc{\p_{k}}{\p_{l}}\neq\overarc{\p_{k'}}{\p_{l'}}}} \Eq_{\frac{1}{2}}\left(\overarc{\p_{k}}{\p_{l}}\cup\overarc{\p_{k'}}{\p_{l'}}\right).
\end{align}

Since $\displaystyle \Eq_{\frac{1}{2}}\left(\overarc{\p_{k}}{\p_{l}}\cup\overarc{\p_{k'}}{\p_{l'}}\right)=\pi_{\frac{1}{2}}\left(\begin{array}{ccc}
\p_{k} &\frown &\p_{l} \\ \hline
\p_{k'} &\frown &\p_{l'} \\ \hline \end{array}\right)=\pi_{\frac{1}{2}}\left(\begin{array}{ccc}
\p_{k'} &\frown &\p_{l'} \\ \hline
\p_{k} &\frown &\p_{l} \\ \hline \end{array}\right)$ and every arc $\begin{array}{ccc}
\p_{k} &\frown &\p_{l} \\ \hline
\p_{k'} &\frown &\p_{l'} \\ \hline \end{array}$ is in exactly one maximal glueing scheme, then every branch of the Wigner caustic is the image of a maximal glueing scheme under the $\frac{1}{2}$-point map $\pi_{\frac{1}{2}}$.

\end{proof}

As a summary of this section we present an algorithm to find all maximal glueing schemes.

\begin{alg}(Finding all maximal glueing schemes of a generic regular closed curve $\M$ parametrized by $f:S^1\to\mathbb{R}^2$)
\begin{enumerate}[(1)]
\item Find the set of local extrema of the angle function $\varphi_{\M}$ of $\M$ (see Definition \ref{DefAngleFunctin} and Definition \ref{DefLocalExtrema}).
\item Find the sequence of local extrema (see Definition \ref{DefSeqOfLocExtr}).
\item Find the sequence of division points (see Definition \ref{DefSeqOfParallPts}).
\item Find the sets of parallel arcs $\Phi_i$ (see Definition \ref{DefSetParallArcs}).
\item Create the following set 
\begin{align*}
\Lambda:=\Big\{&\big\{\overarc{\p_{k_1}}{\p_{l_1}}, \overarc{\p_{k_2}}{p_{l_2}}\big\}: \overarc{\p_{k_1}}{\p_{l_1}}\neq\overarc{\p_{k_2}}{\p_{l_2}}, \\ 
& \exists_i\ \big(\overarc{\p_{k_1}}{\p_{l_1}}\in\Phi_i\wedge \overarc{\p_{k_2}}{\p_{l_2}}\in\Phi_i\big)\vee\big(\overarc{\p_{l_1}}{\p_{k_1}}\in\Phi_i\wedge \overarc{\p_{l_2}}{\p_{k_2}}\in\Phi_i\big)\Big\}.
\end{align*}
\item If there exists a number $k$ such that $\p_k$ is an inflexion point of $\M$ and there exists the set of arcs $\left\{\overarc{\p_{k}}{\p_{l_1}}, \overarc{\p_k}{\p_{l_2}}\right\}$ or $\left\{\overarc{\p_{l_1}}{\p_k}, \overarc{\p_{l_2}}{\p_k}\right\}$ in $\Lambda$, create a glueing scheme $\begin{array}{ccc} \p_k & \frown & \p_{l_1} \\ \hline \p_k &\frown &\p_{l_2} \\ \hline\end{array}$, remove the used set of arcs from $\Lambda$ and go to step $(7)$. Otherwise go to step $(8)$.
\item If the created glueing scheme is of the form $\begin{array}{ccc} \ldots &\frown & \p_{k_1} \\ \hline\ldots &\frown &\p_{l_1}\\ \hline\end{array}$ and there exists the set of arcs $\left\{\overarc{\p_{k_1}}{\p_{k_2}}, \overarc{\p_{l_1}}{\p_{l_2}}\right\}$ or $\left\{\overarc{\p_{k_2}}{\p_{k_1}}, \overarc{\p_{l_2}}{\p_{l_1}}\right\}$ in $\Lambda$, then prolong the scheme to the following scheme $\begin{array}{ccccc}
\ldots &\frown & \p_{k_1} &\frown &\p_{k_2} \\ \hline
\ldots &\frown & \p_{l_1} &\frown &\p_{l_2} \\ \hline \end{array}$, remove the used set of arcs from $\Lambda$ and go to step $(7)$, otherwise the considered glueing scheme is a maximal glueing scheme and then go to step $(6)$.
\item If $\Lambda$ is empty, then all maximal glueing schemes for $\Eq_{\frac{1}{2}}(\M)$ were created, otherwise find any set of arcs $\left\{\overarc{\p_{k_1}}{\p_{l_1}}, \overarc{\p_{k_2}}{\p_{l_2}}\right\}$ in $\Lambda$, create a glueing scheme  $\begin{array}{ccc} \p_{k_1} & \frown & \p_{l_1} \\ \hline \p_{k_2} &\frown &\p_{l_2} \\ \hline\end{array}$, remove the used set of arcs from $\Lambda$ and go to step $(7)$. 
\end{enumerate}
\end{alg}

\subsection{An example of construction of branches of the Wigner caustic}\label{SectionExample}$ $

Let $\M$ be a curve illustrated in Fig. \ref{PictureAlgFindingPhiSets}. Then the sets of parallel arcs are as follows

\begin{align*}
\Phi_0  &=\left\{\overarc{\p_0}{\p_1}, \overarc{\p_4}{\p_5}\right\},\\
\Phi_1 &=\left\{\overarc{\p_1}{\p_2}, \overarc{\p_3}{\p_2}, \overarc{\p_3}{\p_4}, \overarc{\p_5}{\p_0}\right\}.
\end{align*}

Then there exist two maximal glueing schemes of $\M$:
\begin{align}\label{GlueingSchemeEx3}
\begin{array}{ccccccccc} 
\p_0&\frown &\p_1&\frown &\p_2&\frown &\p_3&\frown &\p_4 \\ \hline
\p_4&\frown &\p_5&\frown &\p_0&\frown &\p_5&\frown &\p_0 \\ \hline 
\end{array},\\ \label{GlueingSchemeEx4}
\begin{array}{ccccccc} 
\p_2&\frown &\p_1&\frown &\p_2&\frown &\p_3 \\ \hline
\p_2&\frown &\p_3&\frown &\p_4&\frown &\p_3 \\ \hline 
\end{array}.
\end{align}

By Proposition \ref{PropAlgParityOfCuspsInBranch} the number of cusps of the branch which correspond to (\ref{GlueingSchemeEx3}) is odd.
By Corollary \ref{CorAlgNumberOfInflInEachComponent}  in the glueing scheme (\ref{GlueingSchemeEx3}) there are two parallel pairs containing an inflexion point of $\M$ -- the pairs: 
\begin{align*}
\begin{array}{ccccccccccccccccc} 
\ldots &	\frown	& 	\p_2	&	\frown & \p_3	&	\frown & \ldots\\ \hline
\ldots &	\frown	& 	\p_0	&	\frown & \p_5	&	\frown & \ldots\\ \hline
\end{array}.
\end{align*}
Therefore this branch of the Wigner caustic has exactly two inflexion points -- see Fig. \ref{PictureAlgExEasy}(ii). The same conclusion holds for the glueing scheme (\ref{GlueingSchemeEx4}) and the branch in Fig. \ref{PictureAlgExEasy}(i). In this case we exclude the first and the last parallel pair.

\begin{figure}[h]
\centering
\includegraphics[scale=0.35]{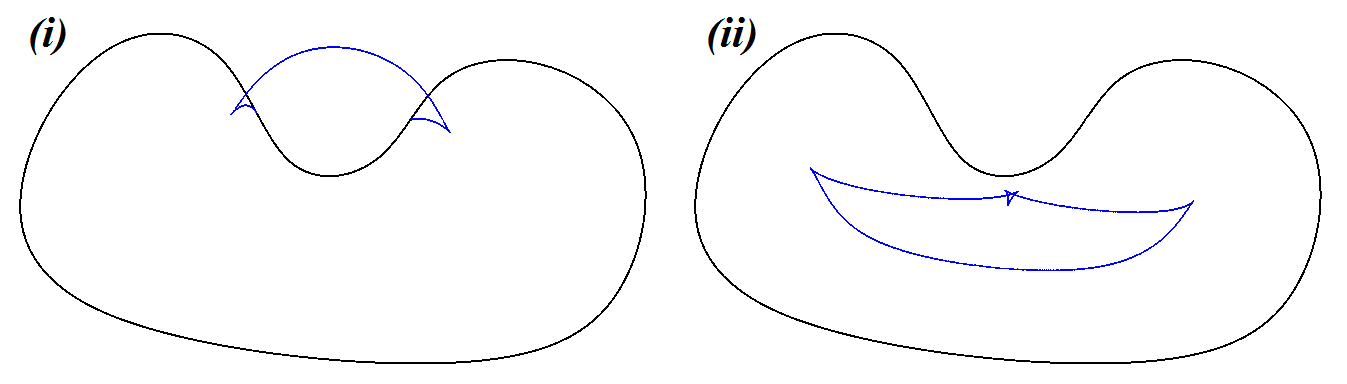}
\caption{A curve $\M$ as in Fig. \ref{PictureAlgFindingPhiSets} and different branches of $\Eq_{\frac{1}{2}}(\M)$}
\label{PictureAlgExEasy}
\end{figure}


\section{The geometry of the Wigner caustic of regular curves}\label{SectionOnShellGeometryOfConvex} 

In this section we start with propositions on numbers of inflexion points and cusp singularities of the Wigner caustic which follows from properties of maximal glueing schemes introduced in Section \ref{SectionAlgorithm}.

\begin{prop}\label{PropAlgPartBetweenIflPt}
Let $\M$ be a generic regular closed curve. If $\M$ has $2n$ inflexion points then there exist exactly $n$ smooth branches of  $\Eq_{\frac{1}{2}}(\M)$ connecting pairs of inflexion points on $\M$ and every inflexion point of $\M$ is the end of exactly one branch of $\Eq_{\frac{1}{2}}(\M)$. Other branches of $\Eq_{\frac{1}{2}}(\M)$ are closed curves.
\end{prop}
\begin{proof}
It is a consequence of Lemma \ref{LemPropMaxGlueSchemes} and Theorem \ref{ThmGlueSchemeIsBranch}.

\end{proof}

\begin{lem}\label{PropCuspsEq}
Let $\C$ be a closed smooth curve with at most cusp singularities. If the rotation number of $\C$ is an integer, then the number of cusps of $\C$ is even and if the rotation number of $\C$ is a half-integer, then the number of $\C$ is odd.
\end{lem}

\begin{proof}

A continuous normal vector field to the germ of a curve with the cusp singularity is directed outside the cusp on the one of two connected regular components and is directed inside the cusp on the other component as it is illustrated in Fig. \ref{PictureNormalVectorToCusp}. That observation end the proof.

\begin{figure}[h]
\centering
\includegraphics[scale=0.16]{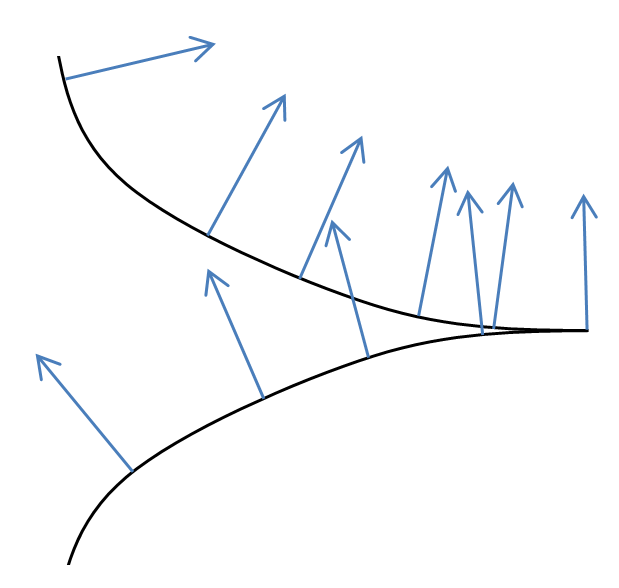}
\caption{A continuous normal vector field at a cusp singularity}
\label{PictureNormalVectorToCusp}
\end{figure}

\end{proof}

\begin{prop}\label{PropAlgParityOfCuspsInBranch}
Let $\M$ be a generic regular closed curve. Let $\mathbbm{n}_{\M}$ be a unit continuous normal vector field to $\M$. Let $\C$ be a smooth branch of $\Eq_{\frac{1}{2}}(\M)$ which does not connect inflexion points. Then the number of cusps of $\C$ is odd if and only if the maximal glueing scheme of $\C$ is in the following form
$\begin{array}{ccccc}
\p_k	&\frown&	 \ldots 	&\frown& \p_l \\ \hline
\p_l 	&\frown&	\ldots 	&\frown& \p_k \\ \hline
\end{array}$ 
and $\mathbbm{n}_{\M}(p_l)=-\mathbbm{n}_{\M}(p_k)$.
\end{prop}
\begin{proof}
If the normal vectors to $\M$ at $\p_{k}$ and $\p_{l}$ are opposite, then the rotation number of $\C$ is equal to $\frac{r}{2}$, where $r$ is an odd integer. By Lemma \ref{PropCuspsEq} the number of cusps in $\C$ is odd. Otherwise the rotation number of $\C$ is an integer, therefore the number of cusps of $\C$ is even.

\end{proof}

By Proposition \ref{PropInflOfEq}, Corollary \ref{PropMisEven} and Proposition \ref{PropAlgPartBetweenIflPt} we get the following corollaries on inflexion points of branches of the Wigner caustic of $\M$.

\begin{cor}\label{CorAlgNumberOfInflInEachComponent}
Let $\M$ be a generic regular closed curve. Let $\C$ be a smooth branch of the Wigner caustic of $\M$. Then the number of inflexion points of $\C$ is equal to the number of parallel pairs containing an inflexion point of $\M$ in the maximal glueing scheme for $\C$ unless they are the beginning or the end of the maximal glueing scheme which connects the inflexion points of $\M$.
\end{cor}

\begin{cor}\label{PropNumOfInflPoints}
Let $\M$ be a generic regular closed curve. Let $2n>0$ be the number of inflexion points of $\M$ and let $\#\mathcal{S}_{\M}=2m$. Then $\Eq_{\frac{1}{2}}(\M)$ has $2m-2n$ inflexion points.
\end{cor}

Now we study the properties of the Wigner caustic on shell, i.e. the branch of the Wigner caustic connecting two inflexion points, see Fig. \ref{PictureAlgExEasy}(i). We are interested in the parity of the number of cusps and the parity of the number of inflexion points on this branch.

\begin{thm}\label{ThmCuspsWCBetweenInfl}
Let $\M$ be a generic regular closed curve. Let $S^1\ni s\mapsto f(s)\in\mathbb{R}^2$ be a parameterization of $\M$, let $f(t_1)$, $f(t_2)$ be inflexion points of $\M$ and let $\C$ be a branch of the Wigner caustic of $\M$ which connects $f(t_1)$ and $f(t_2)$. Then the number of cusps of $\C$ is odd if and only if exactly one of the inflexion points $f(t_1), f(t_2)$ is a singular point of the curve $\C\cup f\big([t_1,t_2]\big)$.
\end{thm}
\begin{proof}
By genericity of $\M$ the points $f(t_1)$ and $f(t_2)$ are ordinary inflexion points of $\M$.

By Corollary 4.8. in \cite{DMR1} we know that the germ of the Wigner caustic at an inflexion point of a generic curve $\M$ together with $\M$ are locally diffeomorphic to the following germ at $(0,0)$:
\begin{align*}
\displaystyle\left\{(p,q)\in\mathbb{R}^2: p=0\right\}\cup\left\{(p,q)\in\mathbb{R}^2: p=-q^2, q\leqslant 0\right\}.
\end{align*}

Let $N=\C\cup f\big([t_1,t_2]\big)$. Then $N$ is a closed curve. The germ of $N$ at $f(t_i)$ for $i=1,2$ is locally diffeomorphic to one of the following germs at $(0,0)$:
\begin{align}\label{halfparabola}
\displaystyle\left\{(p,q)\in\mathbb{R}^2: p=0, q\leqslant 0\right\}\cup\left\{(p,q)\in\mathbb{R}^2: p=-q^2, q\leqslant 0\right\},
\end{align}
\begin{align}\label{halfparabola2}
\displaystyle\left\{(p,q)\in\mathbb{R}^2: p=0, q>0\right\}\cup\left\{(p,q)\in\mathbb{R}^2: p=-q^2, q\leqslant 0\right\}.
\end{align}

In other points $N$ has at most cusp singularities.

Note that the point $(0,0)$ is a singular point of the germ of type (\ref{halfparabola}) and the point $(0,0)$ is a $C^1$-regular point of the germ of type (\ref{halfparabola2}) (see Fig. \ref{HalfParabolaVF}).

Let $\M\ni p\mapsto\mathbbm{n}_{\M}(p)\in S^2$ be a continuous normal vector field to $\M$. Let us assume that the maximal glueing scheme for $C$ has the following form
\begin{align*}
\begin{array}{ccccccccc}
\p_{k_1}&\frown&\p_{k_2}&\frown&\ldots&\frown&\p_{k_{n-1}}&\frown&\p_{k_n}\\ \hline
\p_{l_1}&\frown&\p_{l_2}&\frown&\ldots&\frown&\p_{l_{n-1}}&\frown&\p_{l_n} \\ \hline \end{array},
\end{align*}
where $k_1=l_1, k_n=l_n$. Without loss of generality we can assume that $k_1<k_n$. Let us define a normal vector field $\mathbbm{n}_{N}$ to $N$ as follows:
\begin{itemize}
\item $\mathbbm{n}_{N}(p)=\mathbbm{n}_{\M}(p)$ for $p\in f\big([t_1,t_2]\big)$,
\item $\mathbbm{n}_{N}(p)=\mathbbm{n}_{\M}(a)$ for $p\in C$, where $p=\frac{a+b}{2}$, $a,b$ is a parallel pair of $\M$ such that there exists $i\in\{1, 2, \ldots, n-1\}$ such that 
\begin{align*}
a\in\overarc{\p_{k_i}}{\p_{k_{i+1}}}, b\in\overarc{\p_{l_i}}{\p_{l_{i+1}}}.
\end{align*}
\end{itemize}

\begin{figure}[h]
\centering
\includegraphics[scale=0.22]{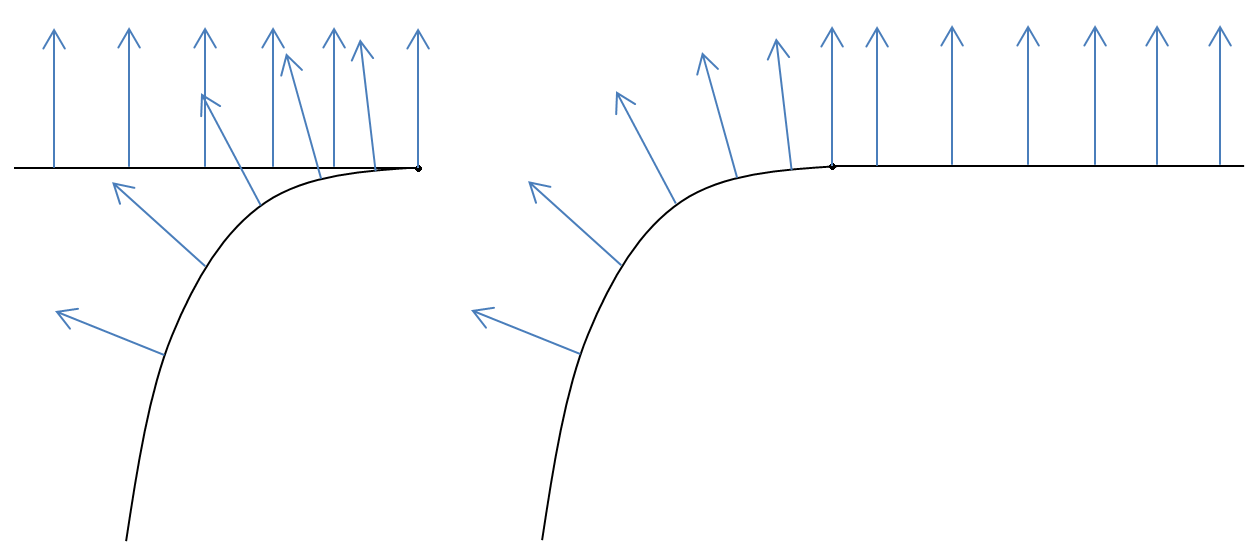}
\caption{A continous normal vector field to the germs of type (\ref{halfparabola}) and (\ref{halfparabola2})}
\label{HalfParabolaVF}
\end{figure}

The vector field $\mathbbm{n}_{N}$ is a continuous unit normal field to $N$. The normal vector field around the points of type (\ref{halfparabola}) and (\ref{halfparabola2}) is described in Fig. \ref{HalfParabolaVF}. Thus by the same argument as in the proof of Lemma \ref{PropCuspsEq} we can get that the total number of cusps and singularities of type (\ref{halfparabola}) in $N$ is even, so the number of cusps of $C$ is odd if and only if exactly one of the inflexion points $f(t_1), f(t_2)$ is of type (\ref{halfparabola}).

\end{proof}

In Fig. \ref{PictureAlgExam}(iii) there is exactly one point of type (\ref{halfparabola}), in Fig. \ref{PictureAlgExEasy}(i) there is an even number of points of type (\ref{halfparabola}).

\begin{prop}\label{PropLimitRatioCurv}
Let $\M$ be a regular curve. Let $(a,b)\ni s\mapsto f(s)\in\mathbb{R}^2$ be a parameterization of $\M$ and let $f(s_0)$ be an ordinary inflexion point of $\M$. Let $t$ be a smooth function-germ on $\mathbb{R}$ at $s_0$ such that $f(s), f(t(s))$ is a parallel pair and $\displaystyle\lim_{s\to s_0}t(s)=s_0$. Let $\kappa_{\M}(s)$ be the curvature of $\M$ at a point $f(s)$. Then
\begin{align}\label{LimitRatioCurv}
\lim_{s\to s_0}\frac{\kappa_{\M}(s)}{\kappa_{\M}(t(s))}=-1.
\end{align}

Furthermore let $C$ be a branch of the Wigner caustic which ends in $f(s_0)$. If \begin{align}\label{f4thderivative}
\det\left[\frac{\mathrm{d}^{4}f}{\mathrm{d}s^{4}}(s_0), \frac{\mathrm{d}f}{\mathrm{d}s}(s_0)\right]\neq 0,
\end{align}
then $\C\cup f\big([s_0,b)\big)$ at $f(s_0)$ is of type (\ref{halfparabola}) if
\begin{align}\label{CondLimitOfDiffOfRatioCurv}
\lim_{s\to s_0}\frac{\mathrm{d}}{\mathrm{d}s}\left(\frac{\kappa_{\M}(s)}{\kappa_{\M}(t(s))}\right)>0
\end{align}
and $\C\cup f\big([s_0,b)\big)$ at $f(s_0)$ is of type (\ref{halfparabola2}) if
\begin{align}\label{CondLimitOfDiffOfRatioCurv}
\lim_{s\to s_0}\frac{\mathrm{d}}{\mathrm{d}s}\left(\frac{\kappa_{\M}(s)}{\kappa_{\M}(t(s))}\right)<0.
\end{align}

\end{prop}
\begin{proof}
Without loss of generality we may assume that locally
\begin{align}\label{fTaylor}
f(s)=(s,F(s)),
\end{align}
where $\displaystyle F(s)=as^3+G(s)$, $a\neq 0$ and $s_0=0$, where $G(s)\in m_1^4$, where $m_n$ is the maximal ideal of smooth function-germs $\mathbb{R}^n\to\mathbb{R}$ vanishing at $0$. Let us notice that $(s, F(s)), (t, F(t))$ is a parallel pair of $\M$ nearby $f(0)$ if and only if $s\neq t$ and
\begin{align*}
F'(s)-F'(t) &=0.
\end{align*}
This is equivalent to
\begin{align*}
(s-t)(3as+3at+H(s,t)) &=0,
\end{align*}
where $H\in m_2^2$ and let $P(s,t)=3as+3at+H(s,t)$. Let $t:(\mathbb{R},0)\to(\mathbb{R},0)$ be a function-germ at $0$ such that 
\begin{align}\label{Pdefinition}
P(s,t(s))=0.
\end{align}
By the implicit function theorem the function-germ $t$ is well defined, because $\displaystyle\frac{\partial P}{\partial t}(0,0)=3a\neq 0$.
By (\ref{Pdefinition}) we get that
\begin{align}
\label{tprimebyP}t'(s) &=-\frac{\displaystyle\frac{\partial P}{\partial s}(s,t)}{\displaystyle\frac{\partial P}{\partial t}(s,t)}.
\end{align}
It implies that
\begin{align}
\label{tprimeofZero}t'(0) &=-1.
\end{align}
Since $F'(s)=F'(t(s))$, then for $s\neq 0$ 
\begin{align}\label{tprimebyKappa}
t'(s) &=\frac{F''(s)}{F''(t(s))}=\frac{\kappa_{\M}(s)}{\kappa_{\M}(t(s))}.
\end{align}
Thus (\ref{LimitRatioCurv}) holds.

\begin{figure}[h]
\centering
\includegraphics[scale=0.25]{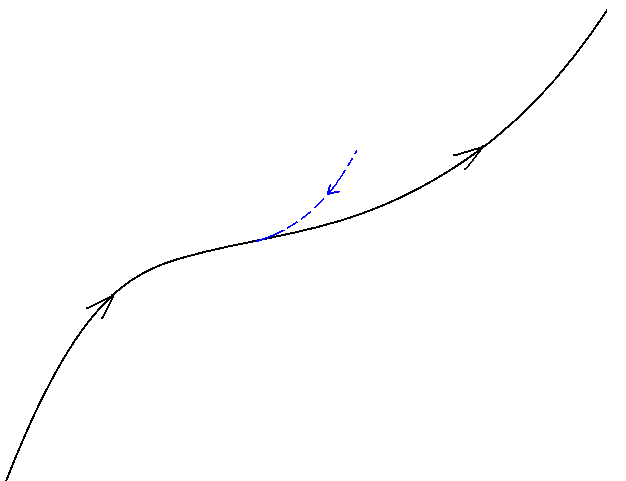}
\caption{A curve $\M$ with an inflexion point and the Wigner caustic of $\M$ (the dashed line)}
\label{InflPointLocal}
\end{figure}

The condition (\ref{f4thderivative}) means that $F^{(4)}(0)\neq 0$. It implies that $\M$ is not locally central symmetric around $f(s_0)=(0,0)$.

The branch of the Wigner caustic which contains $f(0)$ has the following parameterization
\begin{align}
x_{\frac{1}{2}}(s) & = \frac{1}{2}\big(s+t(s), F(s)+F(t(s))\big).
\end{align}
Therefore
\begin{align}\label{wcprime}
x'_{\frac{1}{2}}(s) & = \frac{1}{2}(1+t'(s))\big(1,F'(s)\big).
\end{align}

Since $\C\cup f\big([t_1, t_2]\big)$ at $f(t_1)$ can be only of type (\ref{halfparabola}) or (\ref{halfparabola2}), then $\C\cup f\big([t_1,t_2]\big)$ at $f(t_1)$ is of type (\ref{halfparabola}) if and only if $x_{\frac{1}{2}}'(s)f'(s)<0$ whenever $s\to t_1^-$, therefore by (\ref{wcprime}) we get that $1+t'(s)<0$. By (\ref{tprimeofZero}) we get that $t''(s)>0$ and by (\ref{tprimebyKappa}) we finish the proof.

\end{proof}

\begin{rem}
Under the assumptions of Theorem \ref{PropLimitRatioCurv} if locally $f(s)=(s,F(s))$ then
\begin{align}
\lim_{s\to s_0}\frac{\mathrm{d}}{\mathrm{d}s}\left(\frac{\kappa_{\M}(s)}{\kappa_{\M}(t(s))}\right)=
-\frac{2F^{(4)}(s_0)}{3F^{(3)}(s_0)}.
\end{align}
\end{rem}

\begin{thm}
Let $\M$ be a generic regular closed curve. Let $S^1\ni s\mapsto f(s)\in\mathbb{R}^2$ be a parameterization of $\M$, let $f(s_1)$, $f(s_2)$ be inflexion points of $\M$ and let $\C$ be a branch of the Wigner caustic of $\M$ which connects $f(s_1)$ and $f(s_2)$. Then the number of cusps of $\C$ is odd if and only if 
\begin{align}
\lim_{s\to s_1^\pm}\frac{\mathrm{d}}{\mathrm{d}s}\left(\frac{\kappa_{\M}(s)}{\kappa_{\M}(t_1(s))}\right)\cdot \lim_{s\to s_2^\mp}\frac{\mathrm{d}}{\mathrm{d}s}\left(\frac{\kappa_{\M}(s)}{\kappa_{\M}(t_2(s))}\right)>0,
\end{align}
where $\kappa_{\M}(s)$ denotes the curvature of $\M$ at $f(s)$, the pairs $f(s), f(t_1(s))$ and $f(s), f(t_2(s))$ are parallel pairs such that $t_i(s)\to s_i$ whenever $s\to s_i$ and $s<t_i(s)$ for the left-hand side neighborhood of $s_i$ for $i=1,2$.
\end{thm}
\begin{proof} By genericity of $\M$ we get that $f(s_1)$ and $f(s_2)$ are ordinary inflexion points. Then the theorem is a consequence of Theorem \ref{ThmCuspsWCBetweenInfl} and Proposition \ref{PropLimitRatioCurv}.

\end{proof}

Now we study inflexion points on the Wigner caustic on shell.

\begin{thm}\label{ThmEvenNumberOnShell}
Let $\M$ be a generic regular closed curve. Let $S^1\ni s\mapsto f(s)\in\mathbb{R}^2$ be a parameterization of $\M$ and let $\C$ be a branch of the Wigner caustic which connects two inflexion points $f(t_1)$ and $f(t_2)$ of $\M$. Then the number of inflexion points of $\C$ and the number of inflexion points of the arc $f\big((t_1,t_2)\big)$ are even.
\end{thm}
\begin{proof}
Let $\varphi_{\M}:S^1\to S^1$ be the angle function of $\M$. By genericity of $\M$ all local extrema of $\varphi_{\M}$ are different. Let 
\begin{align*}
\psi_1,\psi_2:[0,T]\to\mbox{graph}\big(\varphi_{\M}\big)\subset S^1\times S^1
\end{align*} 
be the following continuous functions:
\begin{align*}
\psi_1(0) &=\psi_2(0)=\big(t_1,\varphi_{\M}(t_1)\big), \psi_1(T)=\psi_2(T)=\big(t_2, \varphi_{\M}(t_2)\big),\\
\psi_i(t) &=\big(s_i(t), \varphi_{\M}(s_i(t))\big)\ \text{ for }i=1,2,
\end{align*}
where continuous functions $s_1, s_2:[0,T]\to S_1$ satisfy $\varphi_{\M}\big(s_1(t)\big)=\varphi_{\M}\big(s_2(t)\big)$ and $s_1(t)\neq s_2(t)$ for $t\in(0,T)$.

Since $f(t_1)$ is an inflexion point then $\varphi_{\M}(t_1)$ is a local extremum. Without loss of generality we assume that $\varphi_{\M}(t_1)$ is a local minimum. To prove that the number of inflexion points in $f\big((t_1,t_2)\big)$ is even it is enough to show that $\varphi_{\M}(t_2)$ is a local maximum. 

The numbers of local maxima and local minima of $\varphi_{\M}$ are equal. Thus the difference between the number of local maxima and local minima of $\varphi_{\M}\Big|_{S^1-\{t_1\}}$ is one. For small $\varepsilon>0$ the arcs $\psi_1\big|_{[0,\varepsilon]}$ and $\psi_2\big|_{[0,\varepsilon]}$ define the opposite orientations of the graph of $\varphi_{\M}$ and $\varphi_{\M}\circ s_i\big|_{[0,\varepsilon]}$ increases. Let $\varphi_{\M}\big(s_i(\tilde{t})\big)$ for $i=1$ or $i=2$ be a local extremum of $\varphi_{\M}$ such that there are no extremum on $\varphi_{\M}\big(s_1(0,\tilde{t})\big)$ and $\varphi_{\M}\big(s_2(0,\tilde{t})\big)$. Since $\varphi_{\M}(t_1)$ is a local minimum then $\varphi_{\M}(s_i(\tilde{t}))$ is a local maximum and $\psi_j(\tilde{t}-\varepsilon, \tilde{t}+\varepsilon)$ for $j\neq i$ changes the orientation in $\tilde{t}$ (see Fig. \ref{FigAngleFunChange}).

\begin{figure}[h]
\centering
\includegraphics[scale=0.25]{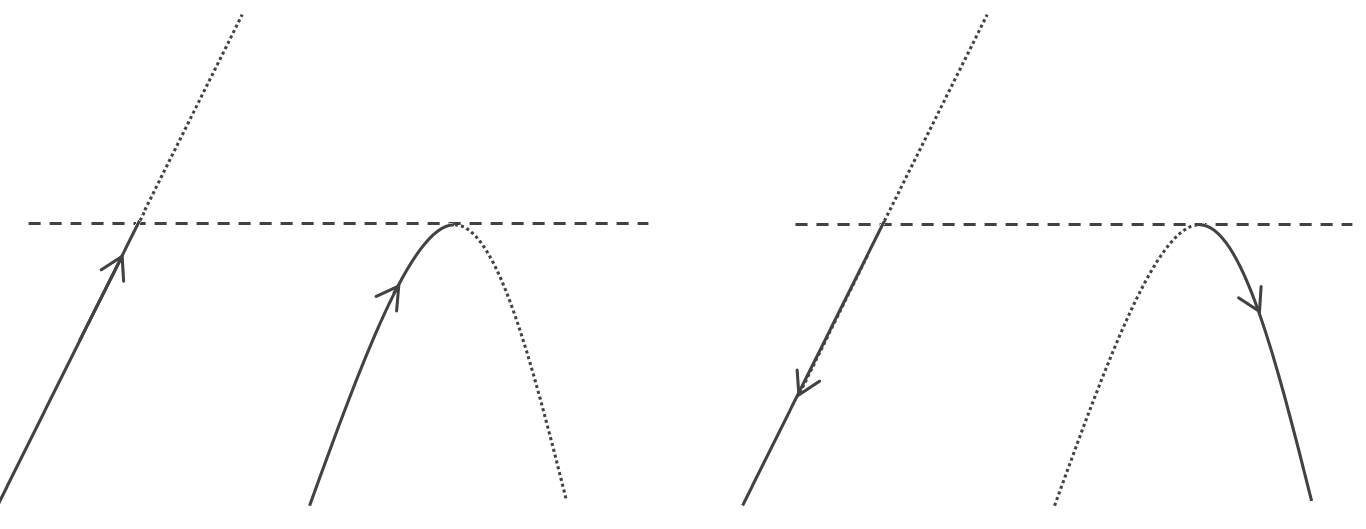}
\caption{A change of the angle function $\varphi_\M$}
\label{FigAngleFunChange}
\end{figure}

The numbers of local maxima and local minima of $\varphi_{\M}\Big|_{S^1-[t_1,\tilde{t}]}$ are equal but the arcs $\psi_1\Big|_{[\tilde{t},\tilde{t}+\varepsilon]}$ and $\psi_2\Big|_{[\tilde{t},\tilde{t}+\varepsilon]}$ define the same orientation of $\mbox{graph}\big(\varphi_{\M}\big)$. Since the function $\varphi_{\M}\Big|_{[\tilde{t},\tilde{t}+\varepsilon]}$ increases then the next extremum is a local minimum. The number of local minima decreases by $1$ and the orientations are opposite after crossing the minimum. Thus the defined orientations are opposite if and only if the difference between the number of local maxima and local minima to cross is one. Since for smal $\varepsilon>0$ the arcs $\psi_1\big|_{[T-\varepsilon,T]}$ and $\psi_2\big|_{[T-\varepsilon,T]}$ define the opposite orientations of the graph of $\varphi_{\M}$, then $\varphi_{\M}(t_2)$ must be a local maximum.

A point $\frac{1}{2}(a+b)$ is an inflexion point of $\C$ if and only if one of points of the parallel pair $a,b$ is an inflexion point of $\M$. 
The number of inflexion points of $\C$ is equal to the sum of the number of changes of the orientations of $\psi_1$ and $\psi_2$ because $\psi_i(\tilde{t}-\varepsilon, \tilde{t}+\varepsilon)$ chagnes the orientation in $\tilde{t}$ if and only if $\varphi_{\M}(s_i(\tilde{t}))$ is a local extremum. Since $\varphi_{\M}(t_1)$ is a minimum and $\varphi_{\M}(t_2)$ is a maximum, then the total number of changes of the orientations is even.

\end{proof}

\begin{figure}[h]
\centering
\includegraphics[scale=0.28]{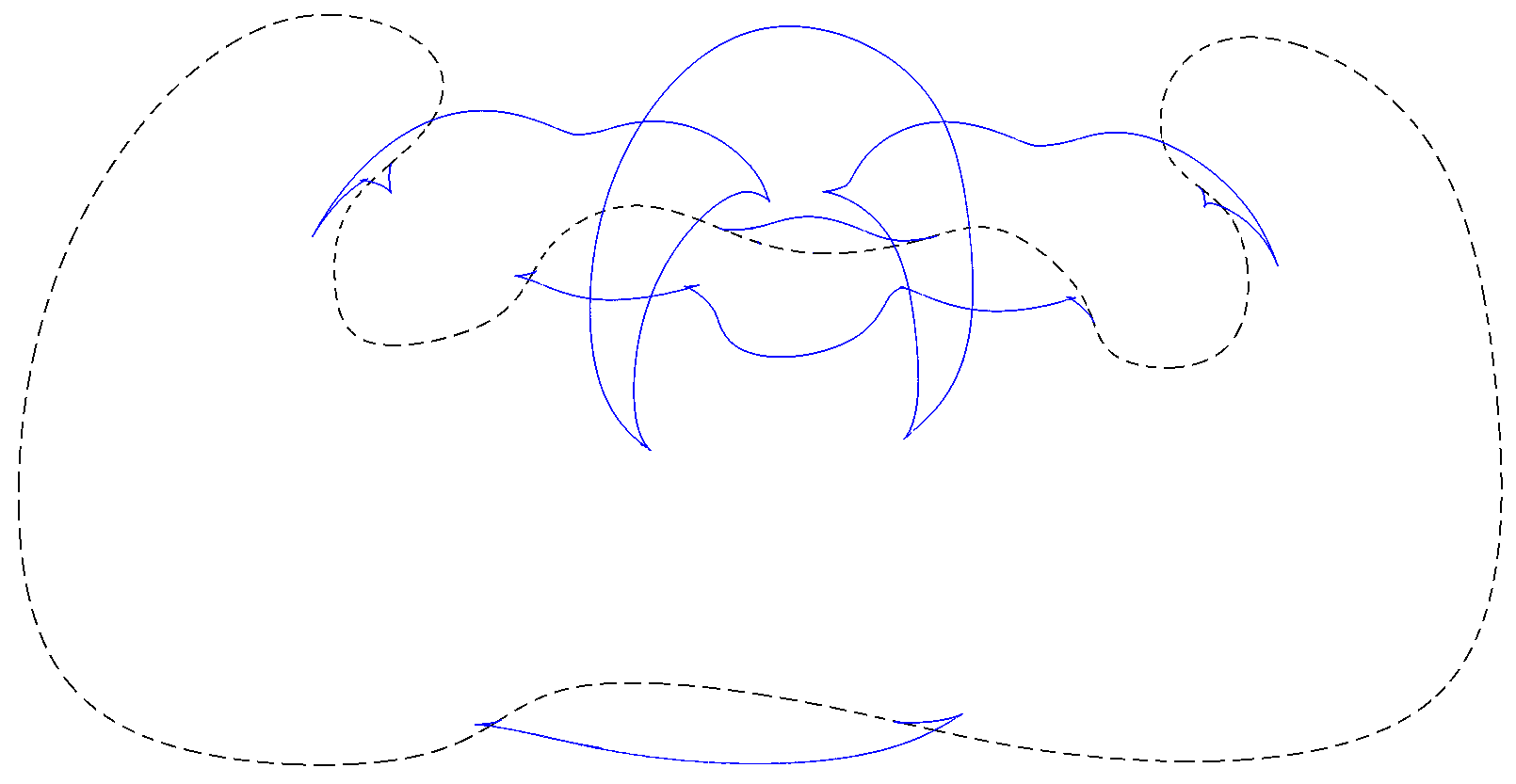}
\caption{A curve $\M$ with $8$ inflexion points (the dashed line) and branches of the Wigner caustic between inflexion points of $\M$}
\label{Fig_4_wc_on_shell}
\end{figure}

In Fig. \ref{Fig_4_wc_on_shell} we illustrate a closed curve $\M$ and branches of the Wigner caustic between inflexion points of $\M$. In Fig. \ref{FigZeroInflPoints} we illustrate a closed curve $\M$ such that the branch of the Wigner caustic which connects two inflexion points of $\M$ has no inflexion points.

\begin{figure}[h]
\centering
\includegraphics[scale=0.4]{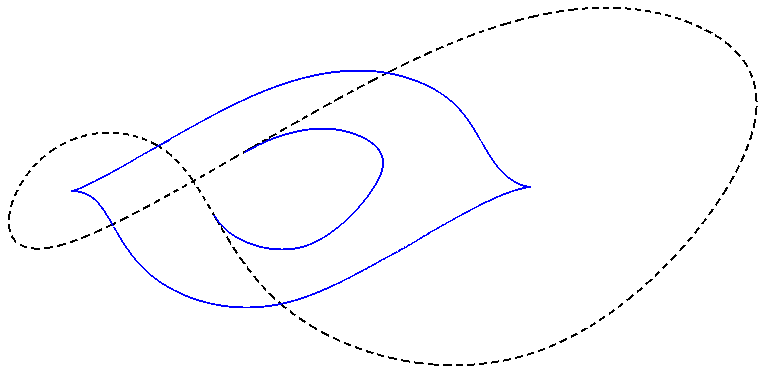}
\caption{A curve $\M$ (the dashed line) and $\Eq_{\frac{1}{2}}(\M)$}
\label{FigZeroInflPoints}
\end{figure}

\begin{lem}\label{LemEvenNumOfInflOfClosedWithCusp}
Let $\C$ be a smooth closed curve with at most cusp singularities. Then the number of inflexion points of $\C$ is even.
\end{lem}
\begin{proof}

If $\C$ is regular, i.e. has no cusp singularities, then by Lemma \ref{LemmaAlgEvenNumOfInfl} we get that $\C$ has an even number of inflexion points.
If $\C$ has cusp singularities, then we change $\C$ nearby each cusp in the way illustrated in Fig. \ref{CuspInflPointsLocal} creating two more inflexion points. After this transformation of $\C$ we obtain a regular closed curve $\widetilde{C}$ such that the parity of the numbers of inflexion points of $\widetilde{C}$ and $C$ are equal. Therefore the number of inflexion points of $\C$ is even.

\begin{figure}[h]
\centering
\includegraphics[scale=0.4]{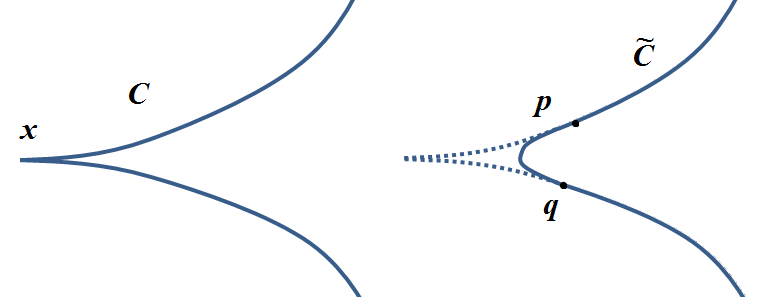}
\caption{A curve $\C$ with the cusp singularity at $x$ and a curve $\widetilde{C}$ with inflexion points at $p$ and $q$}
\label{CuspInflPointsLocal}
\end{figure}

\end{proof}

\begin{prop}
Let $\M$ be a generic regular closed curve. Then the number of inflexion points of each smooth branch of the Wigner caustic of $\M$ is even.
\end{prop}
\begin{proof}
Let us notice that all branches of $\Eq_{\frac{1}{2}}(\M)$ except the branches of the Wigner caustic which connect two inflexion points of  $\M$ are closed curves. So the result for these branches follows from Lemma \ref{LemEvenNumOfInflOfClosedWithCusp}. Otherwise it follows from Theorem \ref{ThmEvenNumberOnShell}.

\end{proof}

\section{The Wigner caustic of closed curves with at most $2$ inflexion points}\label{SectionRosettes}

In this section we study the geometry of the Wigner caustic of closed regular curves with non-vanishing curvature (\textit{rosettes}) and of closed regular curves with exactly two inflexion points.

\begin{defn}
A smooth curve $\gamma: (s_1,s_2)\to\mathbb{R}^2$ is called a \textit{loop} if it is a simple curve with non-vanishing curvature such that $\displaystyle\lim_{s\to s_1^+}\gamma(s)=\lim_{s\to s_2^-}\gamma(s)$. A loop $\gamma$ is called \textit{convex} if the absolute value of its rotation number is not greater than $1$, otherwise it is called \textit{non-convex}.
\end{defn}

\begin{figure}[h]
\centering
\includegraphics[scale=0.27]{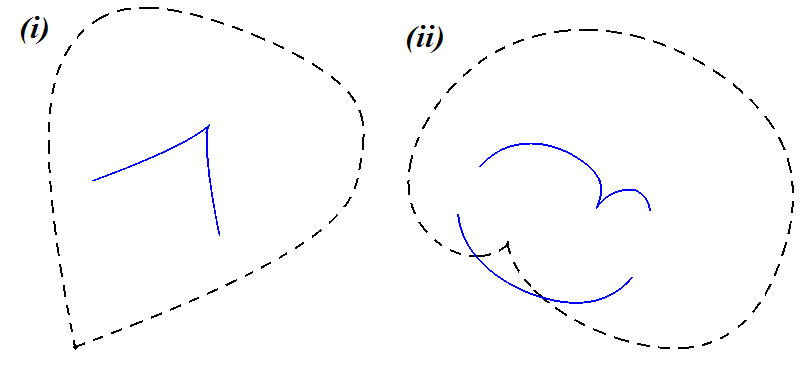}
\caption{(i) A convex loop $L$ (the dashed line) and $\Eq_{\frac{1}{2}}(L)$, (ii) a non-convex loop $L$ (the dashed line) and $\Eq_{\frac{1}{2}}(L)$}
\label{PictureLoops}
\end{figure}

We illustrate examples of loops in Fig. \ref{PictureLoops}.

\begin{thm}[\cite{DZ-singular}]\label{CorWCLoop}
The Wigner caustic of a loop has a singular point.
\end{thm}

\begin{thm}\label{ThmConvex}
Let $C_{n}$ be a generic regular closed parameterized curve with non-vanishing curvature with rotation number equal to $n$. Then
\begin{enumerate}[(i)]
\item the number of smooth branches of $\Eq_{\frac{1}{2}}(C_n)$ is equal to $n$,
\item at least $\displaystyle\left\lfloor\frac{n}{2}\right\rfloor$ branches of $\Eq_{\frac{1}{2}}(C_n)$ are regular closed parameterized curves with non-vanishing curvature,
\item $n-1$ branches of $\Eq_{\frac{1}{2}}(C_n)$ have a rotation number equal to $n$ and one branch has a rotation number equal to $\frac{n}{2}$,
\item every smooth branch of $\Eq_{\frac{1}{2}}(C_n)$ has an even number of cusps if $n$ is even,
\item exactly one branch of $\Eq_{\frac{1}{2}}(C_n)$ has an odd number of cusps if $n$ is odd,
\item cusps of $\Eq_{\frac{1}{2}}(C_n)$ created from loops of $C_n$ are in the same smooth branch of $\Eq_{\frac{1}{2}}(C_n)$,
\item the total number of cusps of $\Eq_{\frac{1}{2}}(C_n)$ is not smaller than $2$,
\end{enumerate}
\end{thm}

\begin{proof}
Since the rotation number of $C_n$ is $n$, for any point $a$ in $C_n$ there exist exactly $2n-1$ points $b\neq a$ such that $a,b$ is a parallal pair of $C_n$. Thus the set of parallel arcs has the following form
$$\Phi_0=\left\{\overarc{\p_0}{\p_1}, \overarc{\p_1}{\p_2}, \ldots, \overarc{\p_{2n-2}}{\p_{2n-1}}, \overarc{\p_{2n-1}}{\p_0}\right\}.$$

Let $\Eq_{\frac{1}{2},k}(C_n)$ be a smooth branch of $\Eq_{\frac{1}{2}}(C_n)$. We can create the following maximal glueing schemes.
\begin{itemize}
\item A maximal glueing scheme of $\Eq_{\frac{1}{2},k}(C_n)$ for $k\in\{1,2,\ldots,n-1\}$:
$$\begin{array}{ccccccccccccc}
\p_0	&\frown&	\p_1	&\frown&	\p_2	&\frown&	\ldots 	&\frown&	\p_{2n-2}	&\frown&	\p_{2n-1} 	&\frown&	\p_0 	\\ \hline 
\p_k	&\frown&	\p_{k+1}	&\frown&	\p_{k+2}	&\frown&	\ldots 	&\frown&	\p_{k-2}	&\frown&	\p_{k-1}	&\frown&	\p_k	\\ \hline
\end{array}.$$
\item A maximal glueing scheme of $\Eq_{\frac{1}{2},n}(C_n)$:
$$\begin{array}{ccccccccccc} 
\p_0	&\frown&	\p_1	&\frown&	\p_2	&\frown&	\ldots 	&\frown&	  \p_{n-1} 	&\frown& \p_n \\ \hline
\p_n	&\frown&	\p_{n+1} &\frown&	\p_{n+2} &\frown&	\ldots 	&\frown&	\p_{2n-1} 	&\frown& \p_0 \\ \hline
\end{array}.$$
\end{itemize}

The total number of arcs of the glueing schemes for the Wigner caustic presented above is $n(2n-1)$. By Proposition \ref{PropNumDiffArcs} the total number of different arcs of the Wigner caustic is equal to the same number. Thus there is no more maximal glueing schemes for the Wigner caustic of $C_n$. 

If $(a_0, a_1, \ldots, a_{2n-1})$ is a sequence of points in $C_n$ with the order compatible with the orientation of $C_n$ such that $a_i, a_j$ is a parallel pair, then $C_n$ is curved in the same side at $a_i$ and $a_j$ if and only if $i-j$ is even. Thus branches $\Eq_{\frac{1}{2},2}(C_n), \Eq_{\frac{1}{2},4}(C_n), \ldots, \Eq_{\frac{1}{2},2\cdot\left\lfloor\frac{n}{2}\right\rfloor}(C_n)$ are created from parallel pairs $a,b$ in $C_n$ such that $C_n$ is curved in the same side at $a$ and $b$ and all the other branches of the Wigner caustic of $C_n$ are created from parallel pairs $a,b$ in $C_n$ such that $C_n$ is curved in the different sides at $a$ and $b$. By Corollaries \ref{PropRegularPointOfEq} and \ref{PropNumOfInflPoints} branches $\Eq_{\frac{1}{2},2}(C_n), \Eq_{\frac{1}{2},4}(C_n), \ldots, \Eq_{\frac{1}{2},2\cdot\left\lfloor\frac{n}{2}\right\rfloor}(C_n)$ are regular closed parameterized curves with non-vanishing curvature. 

By Proposition \ref{PropAlgParityOfCuspsInBranch} the branch $\Eq_{\frac{1}{2},n}(C_n)$ is the only branch of the Wigner caustic of $C_n$ which has an odd number of cusps if $n$ is odd.

We can see that the part of the Wigner caustic created from loops of $C_n$ are all in $\Eq_{\frac{1}{2}, 1}(C_n)$. Every $C_n$ for $n>1$ has at least one loop, so $\Eq_{\frac{1}{2}, 1}(C_n)$ has at least one cusp, but because $\Eq_{\frac{1}{2}, 1}(C_n)$ has an even number of cusps, then $\Eq_{\frac{1}{2}, 1}(C_n)$ has at least two cusps.

\end{proof}

\begin{figure}[h]
\centering
\includegraphics[scale=0.28]{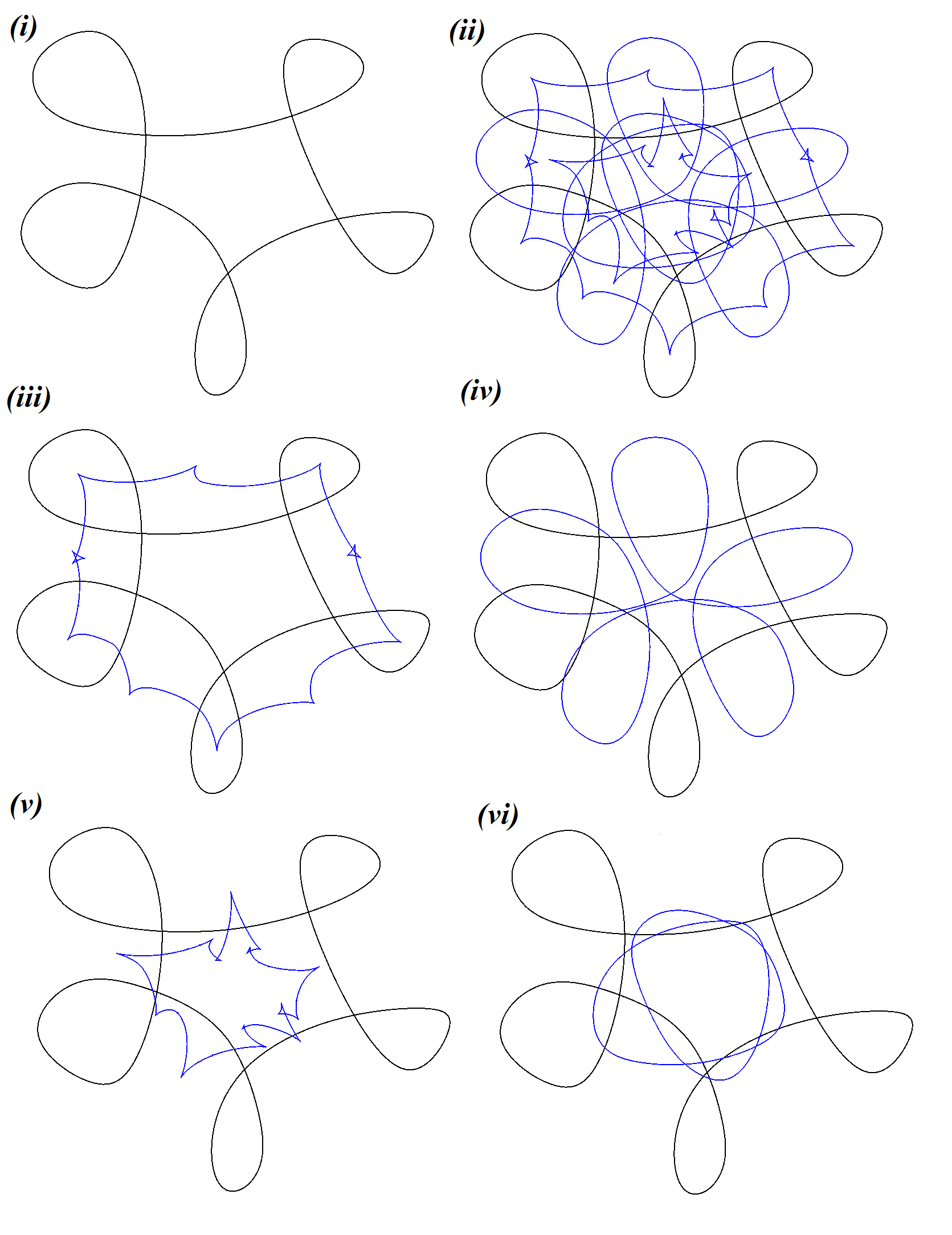}
\caption{(i) A curve $C_4$, (ii) $\Eq_{\frac{1}{2}}(C_4)$, (iii-vi) $C_4$ and different smooth branches of $\Eq_{\frac{1}{2}}(C_4)$}
\label{PictureThmConvex}
\end{figure}

In Fig. \ref{PictureThmConvex}(i) we illustrate a curve of the type $C_4$ and $\Eq_{\frac{1}{2}}(C_4)$. In Fig. \ref{PictureThmConvex}(iii-vi) we illustrate different smooth branches of $\Eq_{\frac{1}{2}}(C_4)$.

\begin{thm}
Let $W_n$ be a generic closed curve with the rotation number $n$. Let $W_n$ has exactly two inflexion points such that one of the arcs of $W_n$ connecting inflexion points is an embedded curve with the absolute value of the rotation number smaller than $\frac{1}{2}$. Then
\begin{enumerate}[(i)]
\item the number of smooth branches of $\Eq_{\frac{1}{2}}(W_n)$ is equal to $n+1$,
\item $n-1$ branches of $\Eq_{\frac{1}{2}}(W_n)$ have a rotation number equal to $n$ and one branch has a rotation number equal to $\frac{n}{2}$,
\item $n-1$ branches of $\Eq_{\frac{1}{2}}(W_n)$ have four inflexion points and two branches has two inflexion points,
\item every smooth branch of $\Eq_{\frac{1}{2}}(W_n)$, except a branch connecting inflexion points of $W_n$, has an even number of cusps if $n$ is even,
\item exactly one smooth branch of $\Eq_{\frac{1}{2}}(W_n)$, except a branch connecting inflexion points of $W_n$, has an odd number of cusps if $n$ is odd,
\item cusps of $\Eq_{\frac{1}{2}}(W_n)$ created from convex loops of $W_n$ are in the same smooth branch of $\Eq_{\frac{1}{2}}(W_n)$.
\end{enumerate}
\end{thm}
\begin{proof} 

\begin{figure}[h]
\centering
\includegraphics[scale=0.30]{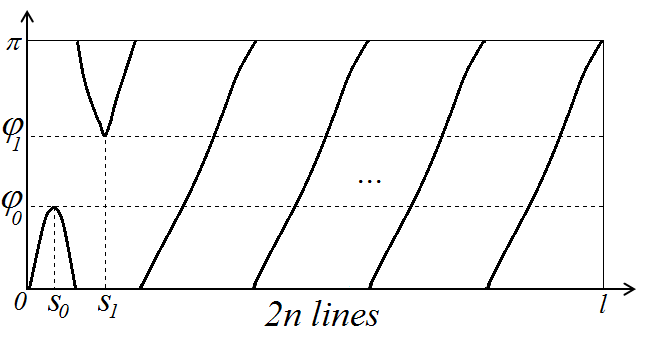}
\caption{An angle function of $W_n$}
\label{PictureAlgWnExample}
\end{figure}

One can notice that the graph of the angle function $\varphi_{W_n}$ has the form presented in Fig. \ref{PictureAlgWnExample}. For that parameterization we get that $f(s_0)$ and $f(s_1)$ corresponds to inflexion points of $W_{n}$ and the sets of the parallel arcs are as follows:
\begin{align*}
\Phi_0 &=\left\{\overarc{\p_2}{\p_3}, \overarc{\p_4}{\p_5}, \overarc{\p_6}{\p_7}, \overarc{\p_8}{\p_9}, \ldots, \overarc{\p_{4n-2}}{\p_{4n-1}}, \overarc{\p_{4n}}{\p_{4n+1}}\right\},\\
\Phi_1 &=\left\{\overarc{\p_0}{\p_1}, \overarc{\p_1}{\p_2}, \overarc{\p_3}{\p_4}, \overarc{\p_5}{\p_6}, \ldots, \overarc{\p_{4n-1}}{\p_{4n}}, \overarc{\p_{4n+1}}{\p_0}\right\}.
\end{align*}

We proceed in the same way like in the proof of Theorem \ref{ThmConvex}.

\end{proof}

An example of a curve $W_1$ and its Wigner caustic are illustrated in Fig. \ref{PictureAlgExEasy}.

\section{The Wigner cuastic of whorls}

In \cite{BO} waves with vacuum wavenumber $k$, travelling in the $\xi$ direction, incident normally on a medium that varies periodically and weakly in the $\eta$ direction were studied. This problem describes the diffraction of light by ultrasound and diffraction of beams of atoms by beams of light and dynamics of quantum particle in optical lattice potential (\cite{CPB}).

In natural dimensionless variables $y=\frac{1}{2}q\eta$, $x=q\sqrt{\frac{n_1}{n_0}}\xi$ (for details see \cite{BO}) the rays regarded as curves $\eta(\xi)$ are described in the following way:
\begin{align*}
y(x, t) &=\sin^{-1}\left[\sin t\,\textrm{sn}\left(x+K(\sin^2t)|\sin^2t)\right)\right],\\ 
p(x, t) &=\dfrac{\mathrm{d} y(x; t)}{\mathrm{d}x}=\sin t\,\mathrm{cn}\left(x+K(\sin^2t)|\sin^2t\right),
\end{align*}
where  $0\leqslant x<\infty$ and $-\frac{\pi}{2}\leqslant t\leqslant\frac{\pi}{2}$, $K(m)$ is the elliptic function, $\textrm{sn}(n|m)$ and $\textrm{cn}(n|m)$ are Jacobi's elliptic sine and Jacobi's cosine functions, respectively. In Figure \ref{PictureWhorlSurface} we illustrate a surface parameterized by
\begin{align}
\label{eq:whorl_surface}\left[0,\dfrac{3\pi}{2}\right]\times\left[-\dfrac{\pi}{2},\dfrac{\pi}{2}\right]\ni (x,t)\mapsto\big(x, y(x, t), p(x, t)\big)\in\mathbb{R}^3.
\end{align}

\begin{figure}[h]
\centering
\includegraphics[scale=0.295]{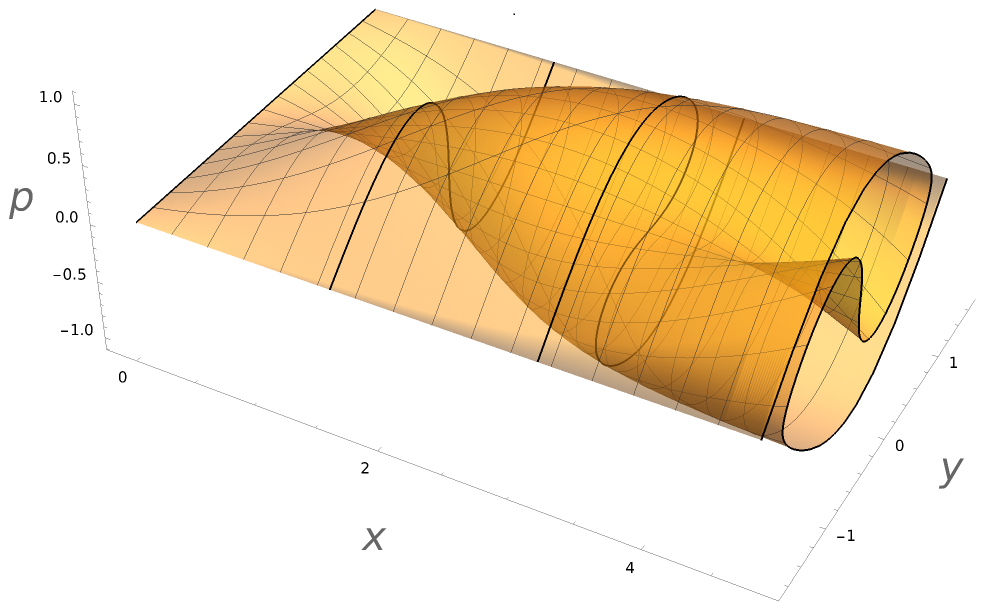}
\includegraphics[scale=0.295]{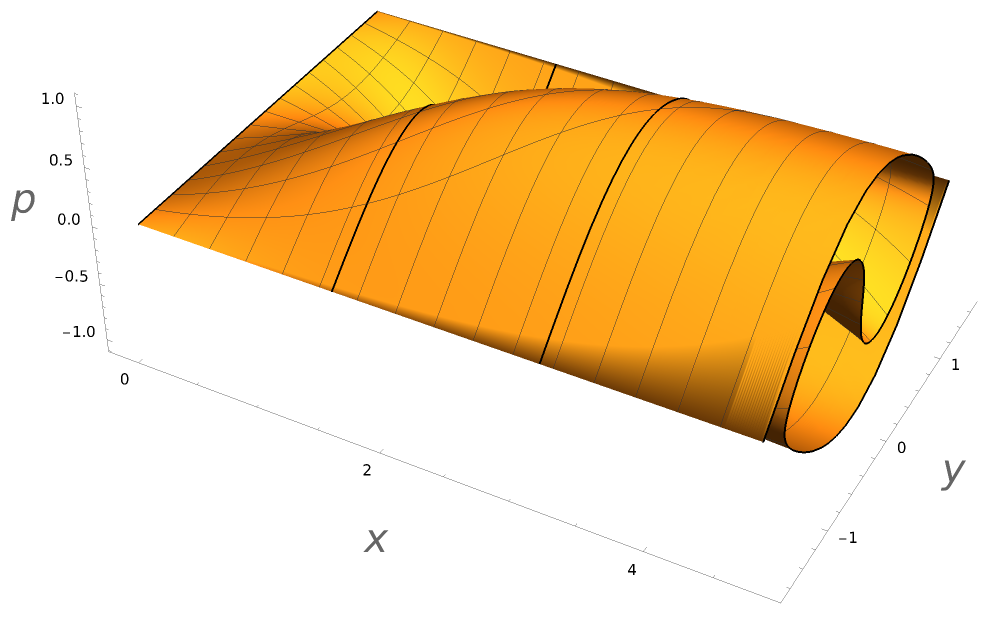}
\caption{The surface parameterized by \eqref{eq:whorl_surface} with different opacities}
\label{PictureWhorlSurface}
\end{figure}

\noindent For a fixed value of $x$ in \eqref{eq:whorl_surface} we obtain so-called \textit{whorl} (\cite{BO}) or \textit{rainbow diagram} (\cite{CPB}) -- see Figure \ref{fig:PictureWorls}.

Catastrophic manifolds of the semiclassical Wigner catastrophes are formed by the Wigner caustic of a fixed whorl and by the whorl by itself (\cite{CPB}). It is worth mentioning that by its construction (\cite{BO}), whorls are $\pi$-periodic in the $y$-value (see Figure \ref{fig:Picture_periodic_whorl}).

We illustrate the Wigner caustic of the periodic whorl from Figure \ref{fig:Picture_periodic_whorl} in Figure \ref{fig:Picture_periodic_whorl_wc}. Notice that the centers of symmetry of the $\pi$-whorl belong to the Wigner caustic.

Now, we explain why the Wigner caustic of the whorl for $x=\pi$ has singular points. We apply a result on existence of singular points of the Wigner caustic (\cite{DZ-singular}).

\pagebreak

\begin{figure}[h!]
\centering
\includegraphics[scale=0.165]{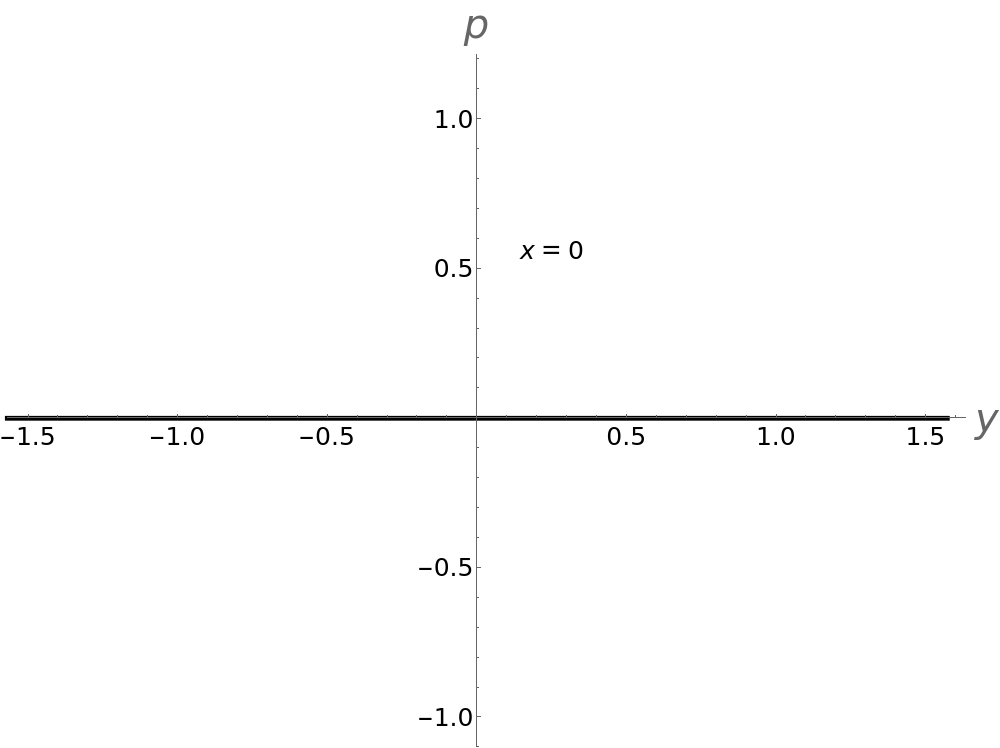}
\includegraphics[scale=0.165]{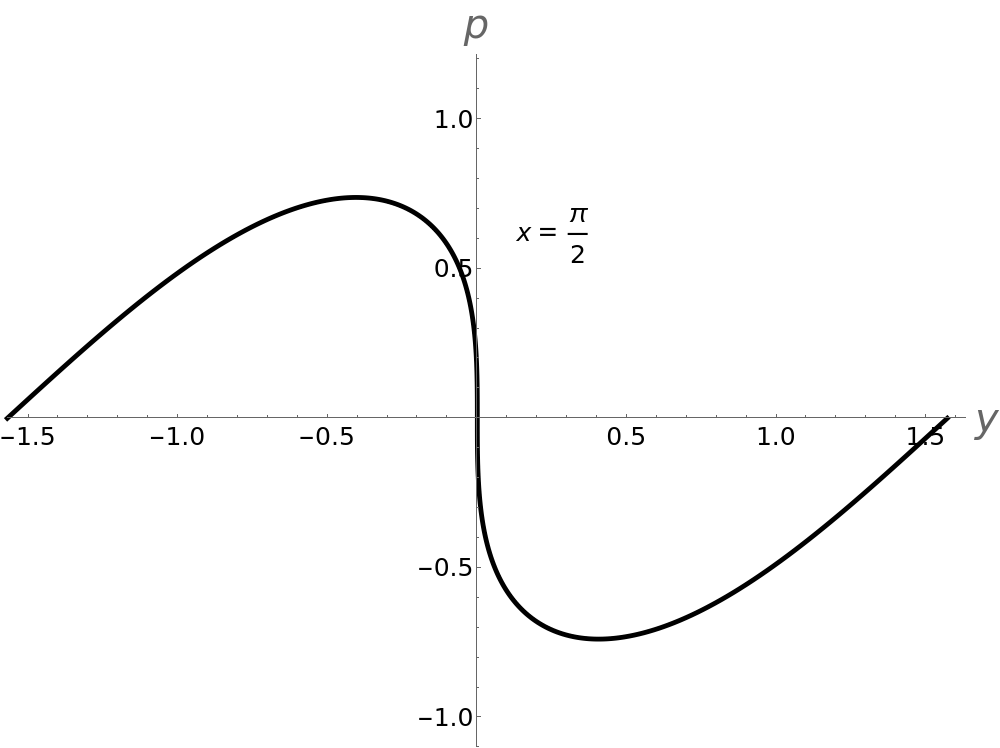}\\
\includegraphics[scale=0.165]{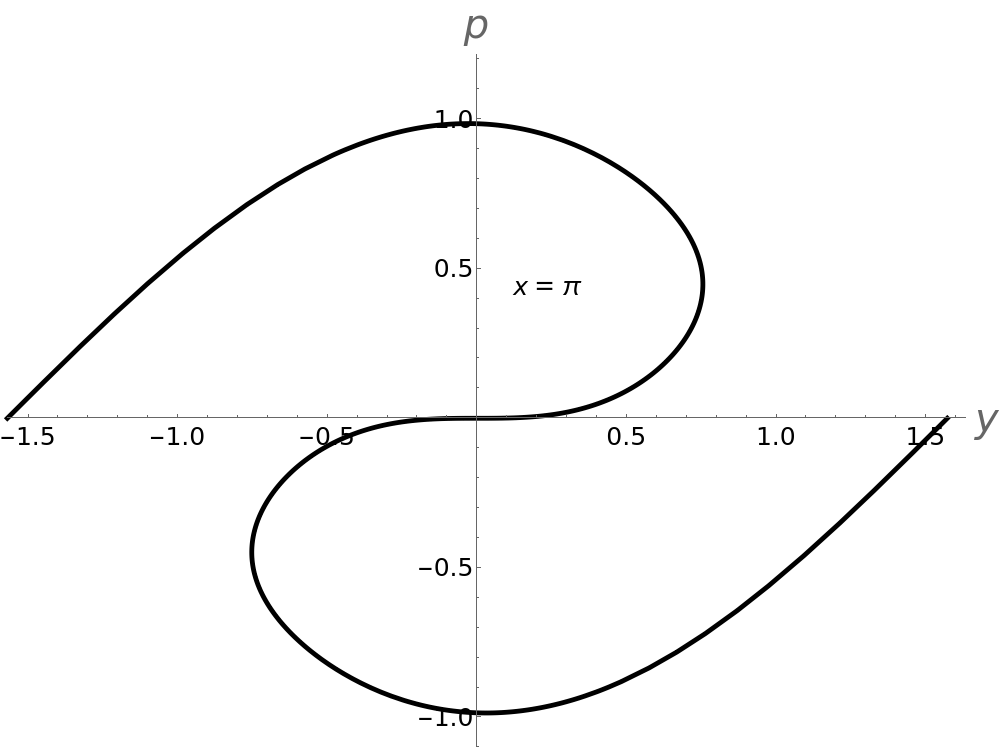}
\includegraphics[scale=0.165]{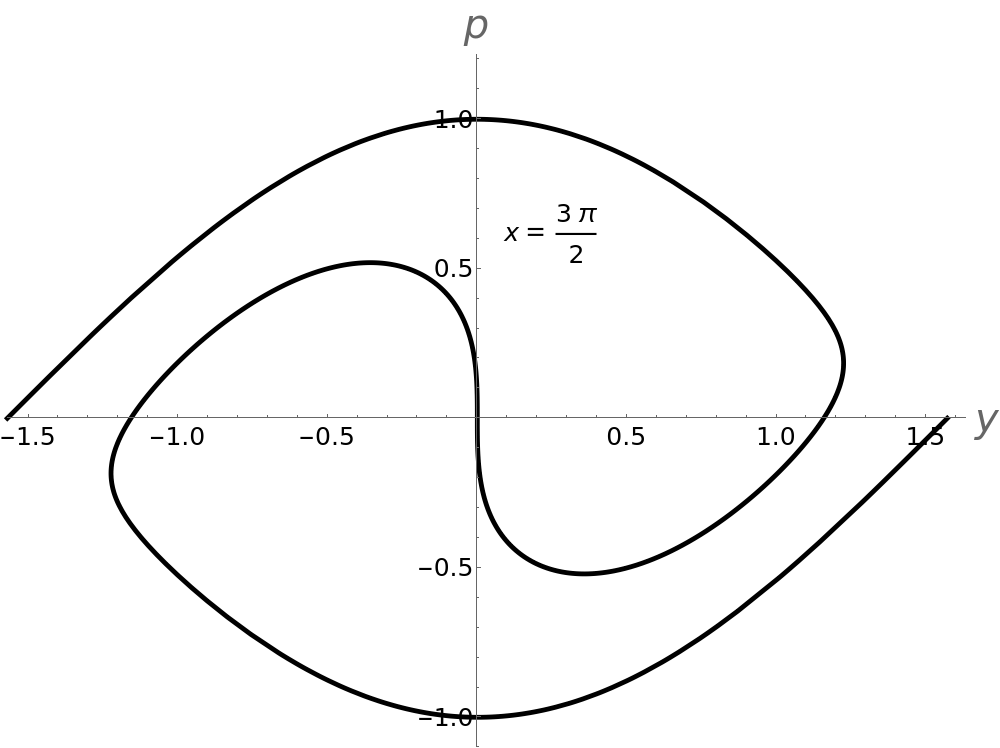}
\caption{Whorls / Rainbow diagrams}
\label{fig:PictureWorls}
\end{figure}

\begin{figure}[h!]
\centering
\includegraphics[width=0.95\textwidth]{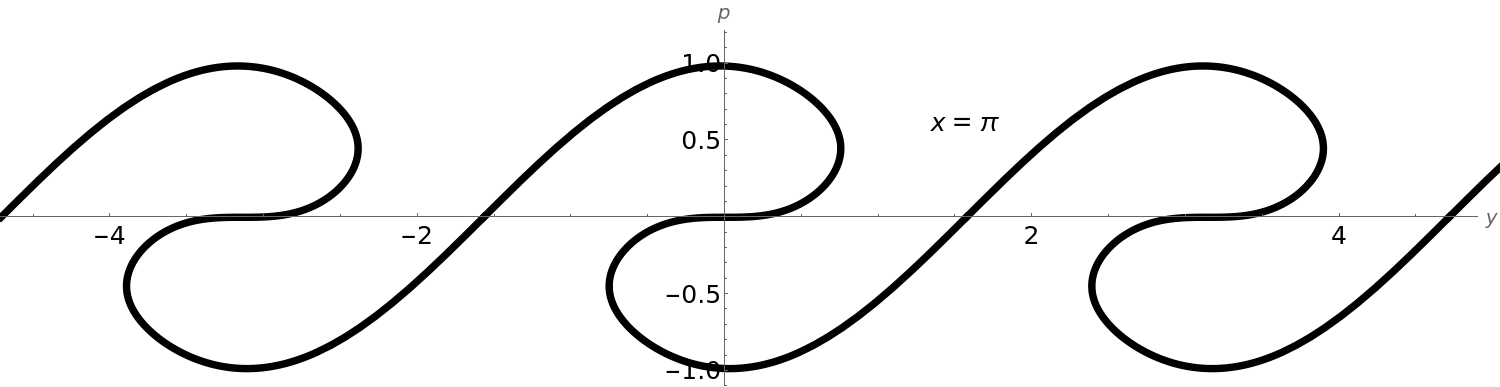}
\caption{The periodic whorl for $x=\pi$}
\label{fig:Picture_periodic_whorl}
\end{figure}

\begin{figure}[h!]
\centering
\includegraphics[width=0.95\textwidth]{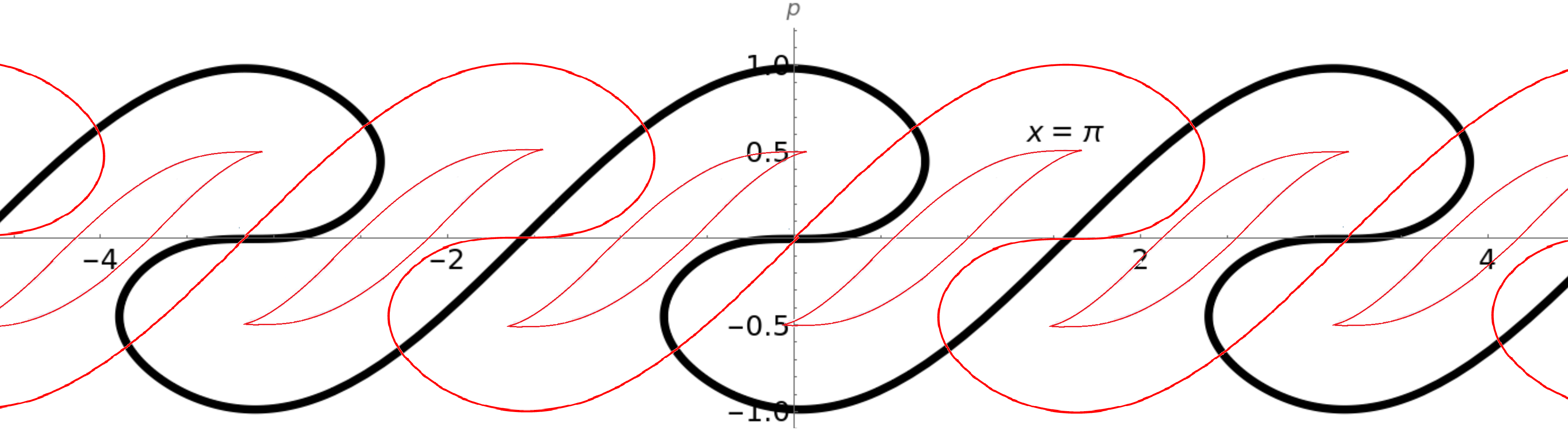}
\caption{The periodic whorl for $x=\pi$ and its Wigner caustic}
\label{fig:Picture_periodic_whorl_wc}
\end{figure}

\pagebreak

\begin{prop}[Proposition 3.7 in \cite{DZ-singular}]\label{prop:SingularWCArcs}
Let $\mathcal{F}_0$ and $\mathcal{F}_1$ be embedded regular curves with endpoints $p$, $q_0$ and $p$, $q_1$, respectively. Let $\ell_0$ be the line through $q_1$ parallel to $T_p\mathcal{F}_0$ and let $\ell_1$ be the line through $q_0$ parallel to $T_p\mathcal{F}_1$. Let $c=\ell_0\cap\ell_1$, $b_0=\ell_0\cap T_p\mathcal{F}_1$, $b_1=\ell_1\cap T_p\mathcal{F}_0$. Let us assume that
\begin{enumerate}[(i)]
\item $T_p\mathcal{F}_0\| T_{q_1}\mathcal{F}_1$ and $T_{q_0}\mathcal{F}_0\| T_p\mathcal{F}_1$,
\item the curvature of $\mathcal{F}_i$ for $i=0,1$ does not vanish at any point,
\item absolute values of rotation numbers of $\mathcal{F}_0$ and $\mathcal{F}_1$ are the same and smaller than $\frac{1}{2}$,
\item for every point $a_i$ in $\mathcal{F}_i$ there is exactly one point $a_j$ in $\mathcal{F}_j$ such that $a_i,a_j$ is a parallel pair for $i\neq j$,
\item $\mathcal{F}_0$, $\mathcal{F}_1$ are curved in the different sides at every parallel pair $a_0,a_1$ such that $a_i\in\mathcal{F}_i$ for $i=0,1$.
\end{enumerate}
Let $\rho_{\textrm{max}}$ (respectively $\rho_{\textrm{min}}$) be the maximum (respectively the minimum) of the set $\displaystyle\left\{\frac{c-b_1}{q_1-b_1}, \frac{c-b_0}{q_0-b_0}\right\}$. 
If $\rho_{\textrm{max}}<1$ or $\rho_{\textrm{min}}>1$, then the Wigner caustic of $\mathcal{F}_0\cup\mathcal{F}_1$ has a singular point.
\end{prop}

\begin{figure}[h]
\centering
\includegraphics[scale=0.225]{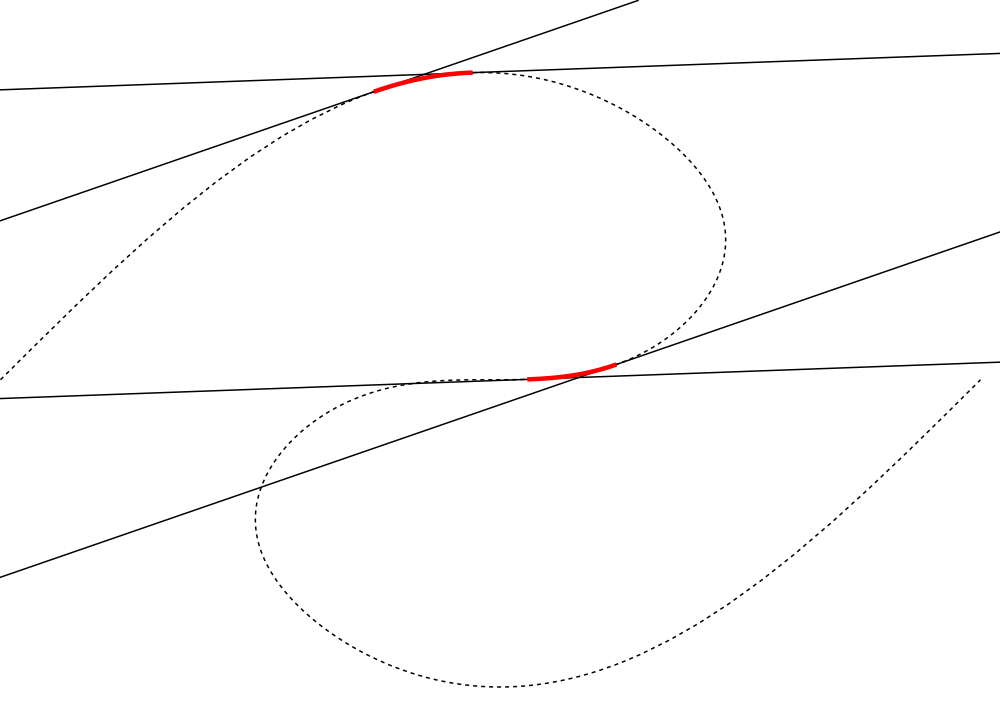}
\caption{The whorl for $x=\pi$ with tangent lines and parallel arcs}
\label{fig:FigureWhorlArcs}
\end{figure}

\begin{figure}[h]
\centering
\includegraphics[scale=0.225]{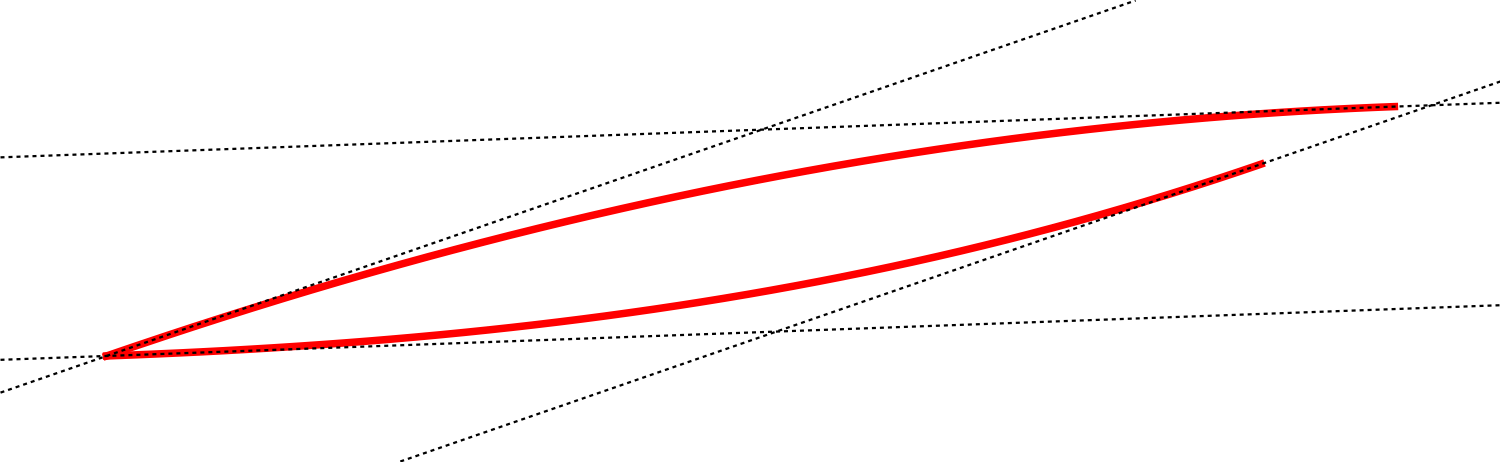}
\caption{Translated parallel arcs from Figure \ref{fig:FigureWhorlArcs}}
\label{fig:FigureWhorlArcsTranslated}
\end{figure}

In Figure \ref{fig:FigureWhorlArcs} we present a $\pi$-whorl with tangent lines for parameters: $t=-0.125$, $t\approx -1.40562$, $t=-0.4$, $t\approx -1.4511$, together with parallel arcs with endpoints at these points. In Figure \ref{fig:FigureWhorlArcsTranslated} we illustrate translated parallel arcs from Figure \ref{fig:FigureWhorlArcs}, which fulfil assumptions of Proposition \ref{prop:SingularWCArcs}. Therefore, the Wigner caustic created from parallel arcs in Figure \ref{fig:FigureWhorlArcs} has a singular point. This method can be applied for other whorls, too.

Furthermore, notice that the tangent lines to the $\pi$-whorl at the points $a_0=(0,0)$, $a=(0,1)$, $b=(0,-1)$ are horizontal, and $a_0$ is an inflexion point of the $\pi$-whorl. Hence, by Proposition \ref{PropInflOfEq} the points $\frac{a_0+a}{2}=(0, 0.5)$ and $\frac{a_0+b}{2}=(0, -0.5)$ are inflexion points of the Wigner caustic of the $\pi$-whorl. These points are nearby singular points of the Wigner caustic of the $\pi$-whorl.

For more figures of the whorls and its Wigner caustics see \cite{CPB}.

\section*{Acknowledgements}
The authors benefitted from the hospitality of the Faculty of Mathematics of the University of Valencia during the preparation of this manuscript. Special thanks to their host, M. Carmen Romero Fuster, for suggesting the subject of this paper and many useful comments. The authors also thank Zbigniew Szafraniec for fruitful discussions and suggestions.

\bibliographystyle{amsalpha}

\end{document}